\documentclass[a4paper,11pt,reqno]{amsart}
\usepackage{amsmath,amsfonts,amsthm,amssymb}
\usepackage{bbm}
\usepackage{dsfont}
\usepackage[foot]{amsaddr}
\usepackage{hyperref}
\usepackage{enumitem}
\usepackage{srcltx}
\theoremstyle{plain}
\newtheorem{lem}{Lemma}[section]
\newtheorem{thm}[lem]{Theorem}
\newtheorem{prop}[lem]{Proposition}
\newtheorem{cor}[lem]{Corollary}
\theoremstyle{definition}

\newtheorem{defn}[lem]{Definition}
\newtheorem{rem}[lem]{Remark}
\newtheorem{ass}[lem]{Assumption}
\numberwithin{equation}{section}
\setlist[itemize,1]{leftmargin=\dimexpr 22pt}
\makeatletter
\def\namedlabel#1#2{\begingroup
    #2%
    \def\@currentlabel{#2}%
    \phantomsection\label{#1}\endgroup
}
\makeatother
\newcommand{\N}{\mathbb{N}}
\newcommand{\Z}{\mathbb{Z}}
\newcommand{\R}{\mathbb{R}}

\newcommand{\J}{\mathbb{J}}
\newcommand{\reg}{\mathrm{r}}
\newcommand{\dif}{\mathrm{d}}
\newcommand{\calf}{\mathcal F}
\newcommand{\calh}{\mathcal H}
\newcommand{\cali}{\mathcal I}
\newcommand{\calj}{\mathcal J}
\newcommand{\cals}{\mathcal S}
\newcommand{\calt}{\mathcal T}
\newcommand{\norm}[1]{\left\Vert#1\right\Vert}
\newcommand{\abs}[1]{\left\vert#1\right\vert}
\newcommand{\eps}{\varepsilon}
\newcommand{\emb}{\hookrightarrow}
\newcommand{\dom}{\operatorname{dom}}
\renewcommand{\geq}{\geqslant}
\renewcommand{\leq}{\leqslant}
\usepackage{footmisc}
\makeatletter
\long\def\@makefntext#1{\indent#1}
\makeatother
\begin{document}
\title[Solutions of Duhamel's formula with a sum of mappings]{Solutions of Duhamel's formula\\with a sum of mappings}

\author{Rados{\l}aw Czaja$^{1}$}
\author{Maria Kania$^{2,*}$}

\address{$^{1,2}$Institute of Mathematics, University of Silesia in Katowice,
Bankowa 14, 40-007 Katowice, Poland.
$^{1}$\textit{E-mail address}: \textup{radoslaw.czaja@us.edu.pl}
$^{2}$\textit{E-mail address}: \textup{maria.kania@us.edu.pl}
\newline
$^1$ORCID: 0000-0003-2245-2916 \quad $^2$ORCID: 0000-0001-8881-9421\quad $^{*}$\textnormal{Corresponding author}
}


\begin{abstract}
Semilinear evolution equations are solved by means of the variation of constants formula 
under the action of smoothing linear operators and a finite number of nonlinearities operating between chosen Banach spaces of a given family. Admissible spaces of initial conditions are determined for the existence and uniqueness of a suitable notion of solution. Its maximal time of existence, short time and blow up time profiles, global extendibility and regularity are analyzed. For extrapolated fractional power scale of Banach spaces the fulfillment of the corresponding Cauchy problem is shown. Chosen applications to the modified viscous Cahn-Hilliard equation in $\R^N$, the evolution of cell population of varying genotype and the semilinear Schr\"odinger equation are presented.   
\end{abstract}

\dedicatory{Dedicated to Professor Jan Cholewa on the occassion of his 60th birthday} 

\vspace*{-3mm}
\maketitle
\section{Introduction} 

In the thirties of the nineteenth century Jean-Marie Duhamel applied the variation of constants formula to the inhomogeneous 
heat transfer equation. This approach is also used to solve semilinear partial differential equations by looking for a fixed point of the corresponding Duhamel formula. The classical way of solving semilinear parabolic equations (see \cite{HE,C-D}) is based on its formulation as an abstract evolution equation 
\begin{equation}\label{e:ROWN}
\dot{u}+Au=\calf(u),\ t>0,
\end{equation}
in a base Banach space $X^0$, with a positive sectorial operator $A$ in $X^0$ in the main linear part and requires that the $X^0$-valued nonlinear part $\calf$ is Lipschitz continuous on bounded subsets of $D(A^\alpha)=X^\alpha$, $\alpha\geq 0$, being a fractional power space corresponding to $A$ provided that the difference of exponents on the scale, equal to $\alpha$, is less than $1$. This requirement 
guarantees that the integral in Duhamel's formula
\begin{equation}\label{e:DUH}
u(t)=S(t)u_0+\int_{0}^{t}S(t-s)\calf(u(s))ds,\ t\in(0,\tau),
\end{equation} 
is well defined, since $-A$ generates a $C^0$ analytic linear semigroup $S(t)=e^{-At}$ for $t\geq 0$ on $X^0$, which immediately smooths elements of $X^0$ to $X^\alpha$ and satisfies the estimate
\begin{equation}\label{e:ANALYTIC}
t^\alpha\norm{S(t)}_{\mathcal{L}(X^0;X^\alpha)}\leq c_\alpha e^{-at},\ t>0,
\end{equation}
with a positive constant $a$. Therefore, for each initial data $u_0$ taken from the phase space $X^\alpha$, there exists a unique local solution $u\in C([0,\tau);X^\alpha)$ of \eqref{e:DUH} starting at $u(0)=u_0$. In fact, the solution has appropriate regularity properties and realizes the abstract differential equation \eqref{e:ROWN} in the space $X^0$. 

This approach can be generalized, since such operators define a~scale of Banach spaces (see e.g. \cite{AM, Las}). In the case of extrapolated scale, the operator $A$ gives rise to Banach spaces $E^\sigma$ for $\sigma\in\J=[-1,\infty)$ (for more complete theory see \cite{AM}), which are isometrically isomorphic to the above fractional power spaces for $\sigma\geq 0$. Thus one can choose $E^\beta$ as the base space and $E^\alpha$ as the phase space with $\alpha,\beta\geq -1$ still such that $\alpha-\beta\in[0,1)$ (cp. \cite{CzD}). 

In the meantime, the authors of \cite{AC1999} considered nonlinearities $\calf$ called $\eps$-regular maps and studied $\eps$-regular solutions, which admitted initial conditions from a broader space than the domain of $\calf$ by taking into consideration the behavior of a~solution to Duhamel's formula near $t=0$. This topic was further pursued in \cite{Q15, CHQRB17}, where the theory of $\eps$-regular solutions was further generalized to solutions of \eqref{e:DUH} with a smoothing linear semigroup $\{S(t)\}_{t\geq 0}$ on a general scale of Banach spaces $\{E^\sigma\}_{\sigma\in\J}$ with nonlinearity acting between $E^\alpha$ and $E^\beta$ with $0\leq\alpha-\beta<1$ and satisfying a specific local Lipschitz condition with growth exponent $\rho\geq 1$,
\begin{equation*}
\norm{\calf(\phi)-\calf(\psi)}_{E^{\beta}}\leq c_0\norm{\phi-\psi}_{E^{\alpha}}(1+\norm{\phi}_{E^{\alpha}}^{\rho-1}+\norm{\psi}_{E^{\alpha}}^{\rho-1}),\ \phi,\psi\in E^{\alpha}.
\end{equation*}
This allowed to determine admissible spaces $E^\gamma$ of initial conditions for which a~reasonable notion of a solution to Duhamel's formula, called in \cite{Q15, CHQRB17} a $\gamma$-solution, can be established by considering its behavior near $t=0$.

In \cite{ACRB1999}, still in the framework of $\eps$-regular solutions, the authors emphasized importance of considering the nonlinearity $\calf$ as a sum of a finite number of $\eps_i$-regular maps $\calf_i$. Such a decomposition better describes the nonlinearity, since we examine the behavior of each part individually within the given family of Banach spaces. 
Moreover, this is a common situation in applications, since the right-hand side of the differential equation reflects different aspects of the studied phenomenon and may take into account friction, air drag, interactions between species, etc., depending on the context. 
 
In this paper we generalize the results of \cite{AC1999, ACRB1999, Q15, CHQRB17}. First, we assume that $\{E^\sigma\}_{\sigma\in\J}$ is a given family of Banach spaces normed by $\|\cdot\|_{E^\sigma}$, where $\J$ is merely a set of indices and not necessarily an interval.  To each space $E^\sigma$, $\sigma\in\J$, we associate a certain real number $\reg(\sigma)$ called \emph{regularity index} and denote $\dif(\sigma,\xi)=\reg(\xi)-\reg(\sigma)$ if $\reg(\xi)\geq\reg(\sigma)$ for $\sigma,\xi\in\J$. Throughout the paper we assume that the spaces are \emph{topologically consistent}, that is, 
\begin{enumerate}
\item[\namedlabel{a:B1}{(A1)}] for every $m\in\N$ and $\sigma_1,\ldots,\sigma_m\in\mathbb{J}$ if $\{u_k\colon k\in\N\} \subset E^{\sigma_1}\cap\ldots\cap E^{\sigma_m}$ and the sequence $(u_k)$ converges in each $E^{\sigma_i}$ then the limit is the same.  
\end{enumerate}
We also distinguish a subset $\J_0$ of $\J$ and consider a family of mappings $S(t)$ for $t>0$ defined on each $E^\sigma$, $\sigma\in\J_0$, in a~consistent way, i.e.,
\begin{enumerate}
\item[\namedlabel{a:B2}{(A2)}] for every $m\in\N$ and $\sigma_1,\ldots,\sigma_m\in\J_0$ if $u\in E^{\sigma_1}\cap\ldots\cap E^{\sigma_m}$ then the value $S(t)u$ coincides for all maps $S(t)$ defined on $E^{\sigma_i}$.  
\end{enumerate}

Under the standing hypotheses \ref{a:B1} and \ref{a:B2}, we consider the situation, as in \cite{ACRB1999}, when the nonlinearity $\calf$ is a sum of a finite number of nonlinear mappings acting between given pairs of spaces from the family $\{E^\sigma\}_{\sigma\in\J}$. Let $\cali$ denote a given nonempty finite set of indices and let $\abs{\cali}\geq 1$ be its cardinality.
We further assume that $\calf_i\colon E^{\alpha_i}\to E^{\beta_i}$ with $\beta_i\in \J_0$, $\alpha_i\in\J$ for $i\in\cali$ satisfy
\begin{equation}\label{e:EPSREG}
\norm{\calf_i(\phi)-\calf_i(\psi)}_{E^{\beta_i}}\leq c_0\norm{\phi-\psi}_{E^{\alpha_i}}(1+\norm{\phi}_{E^{\alpha_i}}^{\rho_i-1}+\norm{\psi}_{E^{\alpha_i}}^{\rho_i-1}),\ \phi,\psi\in E^{\alpha_i},
\end{equation}
for some $c_0>0$ and  $\rho_i\geq 1$ for $i\in\cali$. Moreover, without loss of generality, we assume that $c_0>0$ is such that 
\begin{equation}\label{e:EPSREG2}
\norm{\calf_i(\phi)}_{E^{\beta_i}}\leq c_0(\norm{\phi}_{E^{\alpha_i}}^{\rho_i}+1),\ \phi\in E^{\alpha_i}.
\end{equation}

Our aim is to find solutions of the integral equation corresponding to Duhamel's formula
\begin{equation}\label{e:VCF}
u(t)=S(t)u_0+\sum_{i\in\cali}\int_{0}^{t}S(t-s)\calf_i(u(s))ds\ \text{ for }\ t\in\calt_\tau, 
\end{equation}   
where $u_0$ belongs to $E^\gamma$ for some $\gamma\in\J_0$ to be specified later, and the interval $\calt_\tau$ takes one of the forms:
$\calt_\tau=(0,\tau)$ with $0<\tau\leq\infty$ or $\calt_\tau=(0,\tau]$ with $0<\tau<\infty$. 

In the argument we will exploit a property of the operators $S(t)$, which we call \emph{smoothing for positive times between $E^\sigma$ and $E^\xi$} in resemblance of \eqref{e:ANALYTIC}.

\begin{defn}\label{def:SMOOTHS}
Given $\sigma\in\J_0$, $\xi\in\J$, we say that \emph{the family $\{S(t)\}_{t>0}$ smooths from $E^\sigma$ to $E^\xi$ for positive times}, which we denote 
$\sigma\leadsto\xi$, whenever both following properties hold:
\begin{itemize}
\item[(i)] $S(t)\in\mathcal{L}(E^\sigma;E^\xi)$ for all $t>0$, $\dif(\sigma,\xi)=\reg(\xi)-\reg(\sigma)\geq 0$ and
\begin{equation}\label{e:LINEAREST}
t^{\dif(\sigma,\xi)}\norm{S(t)u}_{E^\xi}\leq M(\sigma,\xi,T)\norm{u}_{E^{\sigma}}\ \text{for all}\ 0<t\leq T<\infty,\ u\in E^{\sigma},
\end{equation}
where $M(\sigma,\xi,T)>0$ is a constant which can be chosen nondecreasing w.r.t. $T$,
\item[(ii)] The map $(0,\infty)\times E^\sigma \ni(t,u)\mapsto S(t)u\in E^\xi$ is continuous.
\end{itemize}
\end{defn} 

To prove continuity of solutions we will need semigroup property of the family $\{S(t)\}_{t>0}$.
\begin{defn}
The family $\{S(t)\}_{t>0}$ is a \emph{semigroup} on $E^{\sigma}$ with $\sigma\in\J_0$ if $S(t)\colon E^{\sigma}\to E^{\sigma}$ for $t>0$ and $S(s)S(t)u=S(s+t)u$ for all $s,t>0$ and $u\in E^{\sigma}$.
\end{defn}

In order to give meaning to the right-hand side of \eqref{e:VCF}, we assume that 
the family $\{S(t)\}_{t>0}$ smooths for positive times from $E^{\beta_i}$ to $E^{\alpha_j}$ for all $i,j\in\cali$, that is,
\begin{equation}\label{e:SETUP1}
\beta_i\leadsto\alpha_j\ \text{ for any }i,j\in\cali.
\end{equation}
and, given $\gamma\in\J_0$, from $E^\gamma$ to $E^{\alpha_j}$ for all $j\in\cali$, i.e., 
\begin{equation}\label{e:SETUP2} 
\gamma\leadsto\alpha_j\ \text{ for any }j\in\cali.
\end{equation}
In particular, we have $\dif(\beta_i,\alpha_j)\geq0$ and $\dif(\gamma,\alpha_j)\geq 0$ for all $i,j\in\cali$.

Following~\cite{RB11,Q15,CHQRB17}, for $\sigma\in\J$, $\theta\geq 0$ and $T>0$, we consider the set
\begin{equation*}
\mathcal{L}^\infty_\theta((0,T];E^\sigma):=\{u\in L^\infty_{loc}((0,T];E^\sigma)\colon \|u\|_{\sigma,\theta,T}:=\sup_{t\in(0,T]}t^\theta\|u(t)\|_{E^\sigma}<\infty\},
\end{equation*}
which is a Banach space with the norm $\|\cdot\|_{\sigma,\theta,T}$.
We use this space to define the notion of a $\gamma$-solution of \eqref{e:VCF} by specifying 
its growth near zero. 

\begin{defn}\label{defn:GAMMASOL}
For $u_0\in E^\gamma$ a function $u\colon\calt_{\tau}\to \bigcap\limits_{i\in\cali}E^{\alpha_i}$ is a $\gamma$-solution of \eqref{e:VCF} if for any $T\in\calt_\tau$ we have 
$u\in \bigcap\limits_{i\in\cali}\mathcal{L}^\infty_{\dif(\gamma,\alpha_i)}((0,T];E^{\alpha_i})=:K(T)$
and $u$ satisfies \eqref{e:VCF} on $\calt_\tau$.
\end{defn}

With the notation
\begin{equation}\label{e:NORMALPHAIT}
\norm{u}_{\alpha_i,T}:=\norm{u}_{\alpha_i,\dif(\gamma,\alpha_i),T}\text{ for }i\in\cali, 
\end{equation}
the space $K(T)$ becomes a~Banach space endowed with the norm
$$\norm{u}_{K(T)}:=\max_{i\in\cali}\|u\|_{\alpha_i,T}\ \text{ for }\ u\in K(T),$$
since for every $i\in\cali$ the space $\mathcal{L}^\infty_{\dif(\gamma,\alpha_i)}((0,T];E^{\alpha_i})$ is a~Banach space and the spaces of the family $\{E^\sigma\}_{\sigma\in\J}$ are topologically consistent by \ref{a:B1}.

We give conditions for the local existence and uniqueness of $\gamma$-solutions of Duhamel's formula and prove the following result in Section~\ref{sec:EXIST}. 

\begin{thm}\label{thm:MAIN1}
Let $\calf_{i}$ satisfy \eqref{e:EPSREG}, \eqref{e:EPSREG2} with $c_0>0$, $\rho_i\geq 1$, $\alpha_i\in\J$, $\beta_i\in\J_0$ for $i\in\cali$ such that \eqref{e:SETUP1} and 
\begin{equation}\label{e:CONDALPHAJ}
\dif(\beta_i,\alpha_j)<1\ \text{ for all}\ i,j\in\cali.
\end{equation}
Assume that $\gamma\in\J_0$ satisfies \eqref{e:SETUP2} and
\begin{equation}\label{e:POSITIVEMUIORAZOMEGAINTRO}
\rho_i\dif(\gamma,\alpha_i)<\min\{1,1+\reg(\beta_i)-\reg(\gamma)\}\ \text{ for all}\ i\in\cali.
\end{equation}
Then, given $R>0$, there exists $0<\tau<\infty$ such that for any $u_0\in E^\gamma$ with $\norm{u_0}_{E^\gamma}\leq R$ there exists a~locally unique $\gamma$-solution $u$ to \eqref{e:VCF} on $(0,\tau]$, i.e.,
$$u\in K(\tau)=\bigcap_{i\in\cali}\mathcal{L}^\infty_{\dif(\gamma,\alpha_i)}((0,\tau];E^{\alpha_i})$$
and $u$ satisfies Duhamel's formula \eqref{e:VCF} on $(0,\tau]$. Moreover, there exists $L>0$ such that if $\norm{u_0}_{E^\gamma}\leq R$, $\norm{v_0}_{E^\gamma}\leq R$ then we have
\begin{equation*}
\norm{u-v}_{K(\tau)}\leq L\norm{u_0-v_0}_{E^{\gamma}},
\end{equation*}
where $u,v$ are $\gamma$-solutions on $(0,\tau]$ corresponding to $u_0,v_0$, respectively. We also have
\begin{equation}\label{e:CONVUTOZERO}
\norm{u}_{K(t)}\to 0\ \text{ as }t\to 0^{+}\ \text{ if }\ \norm{S(\cdot)u_0}_{K(t)}\to 0\ \text{ as }t\to 0^{+}.
\end{equation}
If, additionally, $\{S(t)\}_{t>0}$ is a semigroup on $E^{\gamma}$ and on each $E^{\beta_i}$ for $i\in\cali$, then the above $u$ is unique and
\begin{itemize}
\item[(i)] $u\in\bigcap\limits_{i\in\cali}C((0,\tau];E^{\alpha_i})$,
\item[(ii)] $u\in C((0,\tau];E^\gamma)$ provided that
\begin{equation}\label{e:TOGAMMA}
\gamma\leadsto\gamma\ \text{ and }\ \beta_i\leadsto\gamma\ \text{ for all }i\in\cali,
\end{equation}
\item[(iii)] $u\in C([0,\tau];E^\gamma)$ with $u(0)=u_0$ provided that \eqref{e:TOGAMMA} holds and
\begin{equation}\label{e:STC0}
\lim\limits_{t\to 0^+}\|S(t)u_0-u_0\|_{E^{\gamma}}=0.
\end{equation} 
\end{itemize}
\end{thm}

In Theorem~\ref{thm:MAIN1} we require, in particular, that 
\begin{equation}\label{e:POSITIVEOMEGA}
\omega_i:=1-\rho_i\dif(\gamma,\alpha_i)+\reg(\beta_i)-\reg(\gamma)>0\ \text{ for all}\ i\in\cali,
\end{equation} 
which corresponds to the so-called subcritical case (cp.~\cite{ACRB1999,CHQRB17}). The existence and uniqueness of $\gamma$-solutions in the critical case, when some $\omega_i$'s may vanish, is shown in Theorem~\ref{thm:EXISTCRITICAL}.

Staying in the subcritical case, we then provide an estimate from below for the existence time of a solution (see $\tau(u_0)$ in \eqref{e:time}) and prove that each local $\gamma$-solution extends uniquely to a $\gamma$-solution defined on 
$(0,\tau_{u_0})$ with the maximal time of existence $\tau_{u_0}$. For that purpose we need \eqref{e:TOGAMMA}, which simplifies \eqref{e:POSITIVEMUIORAZOMEGAINTRO} to \eqref{e:POSITIVEOMEGA}. This allows us to study the long time behavior of solutions in the space $E^\gamma$.

\begin{cor}\label{cor:EXIST}
Let $\calf_{i}\colon E^{\alpha_i}\to E^{\beta_i}$, $i\in\cali$, satisfy \eqref{e:EPSREG}, \eqref{e:EPSREG2} and let \eqref{e:SETUP1}, \eqref{e:CONDALPHAJ} hold. Assume that $\gamma\in\J_0$ satisfies \eqref{e:SETUP2}, \eqref{e:TOGAMMA}, \eqref{e:POSITIVEOMEGA} and let $\{S(t)\}_{t>0}$ be a~semigroup on $E^{\gamma}$ and on each $E^{\beta_i}$ for $i\in\cali$. Then for any $u_0\in E^\gamma$ there exists a~unique $\gamma$-solution $u(\cdot,u_0)$ of \eqref{e:VCF} defined on $\calt_{\tau_{u_0}}=(0,\tau_{u_0})$ with the maximal time of existence $\tau_{u_0}\in(\tau(u_0),\infty]$ and satisfying 
\begin{equation}\label{e:HOWREGULAR}
u(\cdot,u_0)\in 
\bigcap_{i\in\cali}C((0,\tau_{u_0});E^\gamma\cap E^{\alpha_i})\cap
\bigcap_{T\in(0,\tau_{u_0})}K(T).
\end{equation}
Moreover, if $\tau_{u_0}<\infty$ then
\begin{equation}\label{e:LIMSUP}
\lim_{t\to\tau_{u_0}^{-}}\norm{u(t,u_0)}_{E^{\gamma}}=\infty.
\end{equation}	
If $u_0\in E^\gamma$ satisfies \eqref{e:STC0} then $u\in C([0,\tau_{u_0});E^\gamma)$ with $u(0)=u_0$.
\end{cor}

In particular, the above setting encompasses extrapolated fractional power scales, see e.g. \cite{AM,Las,Q15,CzD}, which we briefly recall in Section~\ref{sec:FRACTIONAL}. Having a positive sectorial operator $A$ in a reflexive Banach space, the scale $E^\sigma=D((A_{-1})^{\sigma+1})$ for $\sigma\in\J=[-1,\infty)$ consists of reflexive Banach spaces, which are densely and continuously nested, where $A_{-1}$ in $E^{-1}$ denotes the extrapolated operator corresponding to $A$. The operator $-A_{-1}$ generates an analytic $C^{0}$ semigroup $\{S(t)\}_{t\geq0}$ in $E^{-1}$ with $S(0)$ being an identity in $E^{-1}$.   
In this case, the $\gamma$-solution $u(t,u_0)$, $t\in[0,\tau_{u_0})$, of Duhamel's formula \eqref{e:VCF} from Corollary~\ref{cor:EXIST} is differentiable in time and satisfies in $E^\beta$ with $\beta=\min\limits_{i\in\cali}\beta_i$ the differential equation
$$\dot{u}(t)+A_\beta u(t)=\sum_{i\in\cali}\calf_i(u(t)),\ t\in(0,\tau_{u_0}),$$
where $A_\beta$ is the realization of $A_{-1}$ in $E^\beta$; see Theorem~\ref{thm:EXTRAPOLATED} for details. 

In Section~\ref{sec:APPLICATIONS} we present a choice of applications of the above abstract approach to partial differential equations, which are of different types. In extrapolated power scales we solve the modified viscous Cahn-Hilliard equation in $\R^{N}$ 
\begin{equation*}
\dot{u}=(\delta-\Delta)(\Delta u+f(x,u) -\mu \dot{u}),\ t>0,\ x\in\R^{N}
\end{equation*}
with 
\begin{equation*}
f(x,s)=g(x)+m(x)s+\sum_{i=1}^{n}f_i(x,s),\ x\in\R^{N},\ s\in\R,
\end{equation*}
containing a mildly integrable function $m(\cdot)$ and a sum of nonlinear terms $f_i$ of various growth exponents. 

Next example concerns $2m$-th order equation, which describes time evolution of cell population density as a function of the cell genotype 
\begin{equation}\label{e:CELLINTRO}
\dot{u}+(-\Delta)^m u=(G\star f_1(\cdot,u))(x)+f_2(x,u),\ t>0,\ x\in\R^{N},
\end{equation}
with a given kernel $G\in L^{1}(\R^{N})$ in the nonlocal term defined by a convolution and with nonlinear functions $f_i$. We solve Duhamel's formula corresponding to the Cauchy problem for \eqref{e:CELLINTRO} within the family of Lebesgue spaces $E^p=L^p(\R^{N})$, $p\in[1,\infty]$, see Theorem~\ref{thm:CELL}. This family is not nested, but the linear semigroup has abundant $L^p-L^q$ estimates. In Theorem~\ref{thm:CELL2}
we also discuss $\gamma$-solutions for a variant of \eqref{e:CELLINTRO} with the nonlocal term replaced by the Hartree nonlinearity $(G\star|u|^2)u$.

Finally, we apply Theorem~\ref{thm:MAIN1} to the semilinear Schr\"odinger equation
\begin{equation*}
\dot{u}-\mathbbm{i}\Delta u=\sum_{i\in\cali}\calf_i(u),\ t>0,\ x\in\R^{N},
\end{equation*}
again in $L^p(\R^{N})$ spaces. Here the difficulty lies in the scarce available estimates for the linear semigroup. 

We also remark that in the~recent paper \cite{CHRB25} the above setting was considered with a~finite sum of \emph{linear} bounded mappings $\calf_i$ defined in a \emph{common} space, i.e., when $\alpha_i\equiv \alpha$ and $\rho_i\equiv 1$ for all $i\in\cali$.

The content of the paper is as follows. In Section~\ref{sec:PRELIMINARY} we examine the properties of the right-hand side of Duhamel's formula \eqref{e:VCF}. In Section~\ref{sec:EXIST} we prove the main results, Theorem~\ref{thm:MAIN1} and Corollary~\ref{cor:EXIST}, on the existence and uniqueness of $\gamma$-solutions of \eqref{e:VCF}. Section~\ref{sec:REGULARIZATION} contains a scheme to further regularize a $\gamma$-solution. In Section~\ref{sec:FRACTIONAL} we illustrate the theory in extrapolated fractional power scales and prove in Theorem~\ref{thm:EXTRAPOLATED} that in this case $\gamma$-solutions of Duhamel's formula actually solve a differential equation.
The final Section~\ref{sec:APPLICATIONS} is devoted to applications we have just outlined above.   

\section{Preliminary observations}\label{sec:PRELIMINARY}

Under the standing hypotheses  \ref{a:B1}, \ref{a:B2}, we make some preliminary observations concerning the second term on the right-hand side of Duhamel's formula. We recall the constant $M$ from \eqref{e:LINEAREST} and the norm $\norm{\cdot}_{\alpha,T}$ from \eqref{e:NORMALPHAIT}.

\begin{lem}\label{lem:BASIC0}
Let $\calf\colon E^\alpha\to E^\beta$ with $\beta\in\J_0$ and $\alpha\in\J$ satisfy
\begin{equation}\label{e:EPSREGBASIC0}
\norm{\calf(\phi)}_{E^\beta}\leq c_0(\norm{\phi}_{E^\alpha}^\rho+1),\ \phi\in E^\alpha,
\end{equation}
for some $c_0>0$ and $\rho\geq 1$ and assume that $\gamma\in\J$ satisfies 
\begin{equation}\label{e:GAMMAALPHA}
\dif(\gamma,\alpha)=\reg(\alpha)-\reg(\gamma)\geq0.
\end{equation}
Let $\delta\in\J$ be such that 
\begin{equation}\label{e:DELTABASIC0}
\beta\leadsto\delta\ \text{ and }\ \nu(\delta):=1-\dif(\beta,\delta)>0.
\end{equation}
Then, given $u\in\mathcal{L}^\infty_{\dif(\gamma,\alpha)}((0,T];E^{\alpha})$ with $0<T<\infty$, 
\begin{itemize}
\item[(i)]  for $0<t_1<t_2\leq t_3\leq T$ we have
\begin{equation}\label{e:INTEGRALBASIC0}
\int_{t_1}^{t_2}\|S(t_3-s)\calf(u(s))\|_{E^{\delta}}ds\leq c_0M(\beta,\delta,T)\tfrac{(t_2-t_1)^{\nu(\delta)}}{\nu(\delta)}\left(1+\|u\|_{\alpha,T}^{\rho}t_1^{\mu-1}\right),
\end{equation}
where $\mu:=1-\rho\dif(\gamma,\alpha)$.
\item[(ii)] if 
\begin{equation}\label{e:POSITIVEMU}
\mu=1-\rho\dif(\gamma,\alpha)>0,
\end{equation} 
for $0<t_2\leq t_3\leq T$ we have
\begin{equation}\label{e:INTEGRAL1BASIC0}
\int_{0}^{t_2}\|S(t_3-s)\calf(u(s))\|_{E^{\delta}}ds\leq c_0M(\beta,\delta,T)\Bigl(\tfrac{t_2^{\nu(\delta)}}{\nu(\delta)}+\|u\|_{\alpha,T}^{\rho}t_2^{\omega+\reg(\gamma)-\reg(\delta)}B(\nu(\delta),\mu)\Bigr),
\end{equation}
where 
\begin{equation}\label{e:DEFOMEGA}
\omega:=1-\rho\dif(\gamma,\alpha)+\reg(\beta)-\reg(\gamma)
\end{equation} 
and $B(\cdot,\cdot)$ denotes the Euler's Beta function;
$$B(a,b)=B(b,a)=\int_0^1 (1-s)^{a-1}s^{b-1}ds=\int_0^1 s^{a-1}(1-s)^{b-1}ds,\ a,b>0.$$
\end{itemize} 
\end{lem} 

\begin{proof}
By assumption $\delta\in\J$ is such that $\beta\leadsto\delta$, so in particular $\dif(\beta,\delta)\geq 0$. Thus, by \eqref{e:LINEAREST}, \eqref{e:EPSREGBASIC0} and \eqref{e:DELTABASIC0} we get for $0\leq t_1<t_2\leq t_3\leq T$
$$\int_{t_1}^{t_2}\norm{S(t_3-s)\calf(u(s))}_{E^{\delta}}ds\leq  c_0M(\beta,\delta,T)\int_{t_1}^{t_2}(t_3-s)^{\nu(\delta)-1}(1+s^{-\rho\dif(\gamma,\alpha)}\norm{u}_{\alpha,T}^{\rho})ds$$
$$\leq c_0M(\beta,\delta,T)\Bigl(\tfrac{(t_2-t_1)^{\nu(\delta)}}{\nu(\delta)}+\|u\|_{\alpha,T}^{\rho}t_2^{1-\rho\dif(\gamma,\alpha)-\dif(\beta,\delta)}\int_{\frac{t_1}{t_2}}^{1}(1-s)^{\nu(\delta)-1}s^{\mu-1}ds\Bigr).$$
If $t_1>0$ the last integral estimates above by $\frac{t_1^{\mu-1}}{t_2^{\nu(\delta)+\mu-1}}\frac{(t_2-t_1)^{\nu(\delta)}}{\nu(\delta)}$, which leads to \eqref{e:INTEGRALBASIC0}. If $t_1=0$ the last integral is equal to $B(\nu(\delta),\mu)$ provided that \eqref{e:POSITIVEMU} holds, so we get \eqref{e:INTEGRAL1BASIC0}.
\end{proof} 

\begin{lem}\label{lem:BASIC1}
Let $\calf\colon E^\alpha\to E^\beta$ with $\beta\in\J_0$ and $\alpha\in\J$ satisfy
\begin{equation}\label{e:EPSREGBASIC1}
\norm{\calf(\phi)-\calf(\psi)}_{E^\beta}\leq c_0\norm{\phi-\psi}_{E^{\alpha}}(1+\norm{\phi}_{E^{\alpha}}^{\rho-1}+\norm{\psi}_{E^{\alpha}}^{\rho-1}),\ \phi,\psi\in E^{\alpha},
\end{equation}
for some $c_0>0$ and $\rho\geq 1$ and assume that $\gamma\in\J$ satisfies \eqref{e:GAMMAALPHA} and \eqref{e:POSITIVEMU} with $\mu>0$.
Let $\delta\in\J$ be such that \eqref{e:DELTABASIC0} holds with $\nu(\delta)>0$.
Then, given $u,v\in\mathcal{L}^\infty_{\dif(\gamma,\alpha)}((0,T];E^{\alpha})$ with $0<T<\infty$, for $0<t\leq T$ we have
\begin{equation*}
t^{\reg(\delta)-\reg(\gamma)}\int_{0}^{t}\|S(t-s)(\calf(u(s))-\calf(v(s)))\|_{E^{\delta}}ds\leq c_0M(\beta,\delta,T)\kappa(t,\delta)\norm{u-v}_{\alpha,T},
\end{equation*}
where
\begin{equation*}
\kappa(t,\delta):=[t^{1+\reg(\beta)-\reg(\alpha)}+(\|u\|_{\alpha,T}^{\rho-1}+\|v\|_{\alpha,T}^{\rho-1})
t^{\omega}]B(\nu(\delta),\mu)
\end{equation*}
with $\omega$ given in \eqref{e:DEFOMEGA}.
\end{lem}

\begin{proof}
Using \eqref{e:LINEAREST} and \eqref{e:EPSREGBASIC1}, we obtain for $0<t\leq T$
\begin{equation*}
\begin{split}
&\int_{0}^{t}\norm{S(t-s)(\calf(u(s))-\calf(v(s)))}_{E^{\delta}}ds\\
&\leq c_0M(\beta,\delta,T)\int_{0}^{t}(t-s)^{\nu(\delta)-1}\norm{u(s)-v(s)}_{E^{\alpha}}(1+\norm{u(s)}_{E^{\alpha}}^{\rho-1}+\norm{v(s)}_{E^{\alpha}}^{\rho-1})ds\\
&\leq
c_0 M(\beta,\delta,T)\norm{u-v}_{\alpha,T} \int_{0}^{t}(t-s)^{\nu(\delta)-1}\Bigl(s^{-\dif(\gamma,\alpha)}+(\norm{u}_{\alpha,T}^{\rho-1}+\norm{v}_{\alpha,T}^{\rho-1})s^{-\rho\dif(\gamma,\alpha)}\Bigr)ds\\
&\leq
c_0t^{\reg(\gamma)-\reg(\delta)}M(\beta,\delta,T)\norm{u-v}_{\alpha,T}\Big[t^{1+\reg(\beta)-\reg(\alpha)}+(\norm{u}_{\alpha,T}^{\rho-1}+\norm{v}_{\alpha,T}^{\rho-1})t^{\omega}\Big]B(\nu(\delta),\mu),
\end{split}
\end{equation*}
where in the last inequality we used \eqref{e:DELTABASIC0}, \eqref{e:POSITIVEMU} and the fact that $\rho\geq 1$.
\end{proof}

We denote the right-hand side of \eqref{e:VCF} by
\begin{equation}\label{e:DEFH}
H(u,u_0)(t)=S(t)u_0+\sum_{i\in\cali}\int_{0}^{t}S(t-s)\calf_i(u(s))ds,\ t\in\calt_\tau,
\end{equation}
for $u\in\bigcap\limits_{T\in\calt_\tau}\bigcap\limits_{i\in\cali}\mathcal{L}^\infty_{\dif(\gamma,\alpha_i)}((0,T];E^{\alpha_i})$ and $u_0\in E^\gamma$. Following \cite[Lemmas 2.3, 2.4]{CHQRB17}, we examine its growth and regularity. 

\begin{lem}
\label{lem:REGH}
Let \eqref{e:EPSREG2} hold with $\beta_i\in\J_0$, $\alpha_i\in\J$ and $\rho_i\geq 1$ for $i\in\cali$ and assume that $\gamma\in\J_0$ satisfies 
\begin{equation}\label{e:GAMMAALPHAI}
\dif(\gamma,\alpha_i)=\reg(\alpha_i)-\reg(\gamma)\geq0\ \text{ for all}\ i\in\cali,
\end{equation}
\begin{equation}\label{e:POSITIVEMUI}
\mu_i:=1-\rho_i\dif(\gamma,\alpha_i)>0\ \text{ for all}\ i\in\cali
\end{equation} 
and
\begin{equation}\label{e:NONNEGATIVEOMEGA}
\omega_i=1-\rho_i\dif(\gamma,\alpha_i)+\reg(\beta_i)-\reg(\gamma)\geq 0\ \text{ for all}\ i\in\cali.
\end{equation}
Let $T\in\calt_\tau$, $u\in\bigcap\limits_{i\in\cali}\mathcal{L}^\infty_{\dif(\gamma,\alpha_i)}((0,T];E^{\alpha_i})$ and $u_0\in E^\gamma$. 

\noindent 
$(i)$ If $\delta\in\J$ is such that 
\begin{equation}\label{e:TODELTA1}
\gamma\leadsto\delta
\end{equation}
and
\begin{equation}\label{e:THETA}
\beta_i\leadsto\delta\ \text{ and }\ \nu_i(\delta):=1-\dif(\beta_i,\delta)>0\ \text{ for all}\ i\in\cali,
\end{equation}
then we have $H(u,u_0)\in\mathcal{L}^\infty_{\dif(\gamma,\delta)}((0,T];E^{\delta})$ and for $t\in(0,T]$
\begin{equation}\label{e:OSZHINL}
\begin{split}
t^{\dif(\gamma,\delta)}\|H(u,u_0)(t)\|_{E^{\delta}}&\leq t^{\dif(\gamma,\delta)}\norm{S(t)u_0}_{E^{\delta}}\\
&+c_0\sum_{i\in\cali}M(\beta_i,\delta,T)\Big[\tfrac{t^{\nu_i(\delta)+\dif(\gamma,\delta)}}{\nu_i(\delta)}+\|u\|_{\alpha_i,T}^{\rho_i}t^{\omega_i}B(\nu_i(\delta),\mu_i)\Big].
\end{split}
\end{equation}
$(ii)$ If \eqref{e:SETUP1}, \eqref{e:SETUP2} and \eqref{e:CONDALPHAJ} hold, 
then $H(u,u_0)\in\bigcap\limits_{i\in\cali}\mathcal{L}^\infty_{\dif(\gamma,\alpha_i)}((0,T];E^{\alpha_i})$.
\end{lem}

\begin{proof}
By Lemma~\ref{lem:BASIC0}~(ii)  we have \eqref{e:OSZHINL} for $t\in(0,T]$.
By \eqref{e:NONNEGATIVEOMEGA} and \eqref{e:TODELTA1} we get
\begin{equation*}
\begin{split}
\norm{H(u,u_0)}_{\delta,T}&\leq M(\gamma,\delta,T)\norm{u_0}_{E^\gamma}\\
&+c_0\sum_{i\in\cali}M(\beta_i,\delta,T)\Big[\tfrac{T^{\nu_i(\delta)+\dif(\gamma,\delta)}}{\nu_i(\delta)}+\|u\|_{\alpha_i,T}^{\rho_i}T^{\omega_i}B(\nu_i(\delta),\mu_i)\Big],
\end{split}
\end{equation*}
which shows that $H(u,u_0)\in\mathcal{L}^\infty_{\dif(\gamma,\delta)}((0,T];E^{\delta})$.
\end{proof}

\begin{lem}\label{lem:REGH2}
Let \eqref{e:EPSREG2} hold with $\beta_i\in\J_0$, $\alpha_i\in\J$ and $\rho_i\geq 1$ for $i\in\cali$ and assume that
$\gamma\in\J_0$ satisfies \eqref{e:GAMMAALPHAI} and \eqref{e:POSITIVEMUI}.  
Assume also that $\{S(t)\}_{t>0}$ is a semigroup on $E^{\gamma}$ and on each $E^{\beta_i}$, $i\in\cali$,
and let $u\in\bigcap\limits_{T\in\calt_\tau}\bigcap\limits_{i\in\cali}\mathcal{L}^\infty_{\dif(\gamma,\alpha_i)}((0,T];E^{\alpha_i})$ and $u_0\in E^\gamma$.
If $\delta\in\J$ satisfies \eqref{e:TODELTA1} and \eqref{e:THETA}, 
then $H(u,u_0)\in C(\calt_\tau;E^{\delta})$.
\end{lem}

\begin{proof} 
Let $0<t_0<T\in\calt_\tau$ and observe that for $\frac{t_0}{2}\leq t\leq T$ we have
\begin{equation*}
\|H(u,u_0)(t)-H(u,u_0)(t_0)\|_{E^{\delta}}\leq\|\left(S\left( t-\tfrac{t_0}{2}\right)-S\left(\tfrac{t_0}{2}\right)\right) S\left(\tfrac{t_0}{2}\right)u_0\|_{E^{\delta}}+I_t,
\end{equation*}
where 
\begin{equation}\label{e:IT}
I_t=\Bigl\|\sum_{i\in\cali}\int_{0}^{t}S(t-s)\calf_i(u(s))ds-\sum_{i\in\cali}\int_{0}^{t_0}S(t_0-s)\calf_i(u(s))ds\Bigr\|_{E^{\delta}}.
\end{equation}
Since $S\left(\frac{t_0}{2}\right)u_0\in E^{\gamma}$ and by \eqref{e:TODELTA1} the semigroup smooths from $E^\gamma$ to $E^{\delta}$, we have 
$$\norm{\left(S\left( t-\tfrac{t_0}{2}\right)-S\left(\tfrac{t_0}{2}\right)\right) S\left(\tfrac{t_0}{2}\right)u_0}_{E^{\delta}}\to 0\ \text{ as }\ t\to t_0.$$
To prove the claim, it suffices to show that $I_t$ in \eqref{e:IT} tends to zero as $t\to t_0$.

We first show that it converges to zero as $t\to t_0^{+}$. For $t_0<t\leq T$ we get
$$I_t\leq\sum_{i\in\cali}\int_{t_0}^{t}\norm{S(t-s)\calf_i(u(s))}_{E^{\delta}}ds+
\int_{0}^{t_0}g_t(s)ds$$
with 
$$g_t(s):=\Bigl\|\sum_{i\in\cali}(S(t-s)-S(t_0-s))\calf_i(u(s))\Bigr\|_{E^{\delta}}\ \text{ for }\ s\in(0,t_0).$$
By \eqref{e:INTEGRALBASIC0} with $\mu_i, \nu_i(\delta)\in(0,1]$ given in \eqref{e:POSITIVEMUI} and \eqref{e:THETA}, respectively, we know that 
\begin{equation*}
\begin{split}
\sum_{i\in\cali}&\int_{t_0}^{t}\norm{S(t-s)\calf_i(u(s))}_{E^{\delta}}ds\leq c_0\sum_{i\in\cali}M(\beta_i,\delta,T)\tfrac{(t-t_0)^{\nu_i(\delta)}}{\nu_i(\delta)}
\Bigl(1+\|u\|_{\alpha_i,T}^{\rho_i}t_0^{\mu_i-1}\Bigr),
\end{split}
\end{equation*}
which tends to zero as $t\to t_0^+$. To show that
\begin{equation}\label{e:TOLDCT}
\int_{0}^{t_0}g_t(s)ds\to 0\ \text{ as }\ t\to t_0^{+},
\end{equation} 
using \eqref{e:EPSREG2} and \eqref{e:THETA} with $\nu_i(\delta)\in(0,1]$, we estimate for $s\in(0,t_0)$ 
\begin{equation*}
\begin{split}
g_t(s)&\leq\sum_{i\in\cali}\norm{S(t-s)\calf_i(u(s))}_{E^{\delta}}+\sum_{i\in\cali}\norm{S(t_0-s)\calf_i(u(s))}_{E^{\delta}}\\
&\leq 2c_0\sum_{i\in\cali}M(\beta_i,\delta,T)(t_0-s)^{\nu_i(\delta)-1}\Bigl(1+\norm{u}_{\alpha_i,T}^{\rho_i}s^{\mu_i-1}\Bigr)=:g(s).
\end{split}
\end{equation*}
Hence $g_t$ is bounded by $g$ on $(0,t_0)$ for $t_0<t\leq T$. Moreover, by \eqref{e:POSITIVEMUI} we have 
\begin{equation*}
\int_{0}^{t_0}g(s)ds\leq 2c_0\sum_{i\in\cali}M(\beta_i,\delta,T)\Bigl(\tfrac{t_0^{\nu_i(\delta)}}{\nu_i(\delta)}+ \norm{u}_{\alpha_i,T}^{\rho_i}t_0^{\nu_i(\delta)+\mu_i-1}B(\nu_i(\delta),\mu_i)\Bigr).
\end{equation*}
For a given $s\in(0,t_0)$ we have
\begin{equation*}
0\leq g_t(s)\leq
\sum_{i\in\cali}\norm{(S(t-s)-S(t_0-s))\calf_i(u(s))}_{E^{\delta}}\to 0\ \text{ as }\ t\to t_0^+,
\end{equation*}
since $\calf_i(u(s))\in E^{\beta_i}$ and the semigroup smooths from $E^{\beta_i}$ to $E^{\delta}$.
Consequently, applying Lebesgue's dominated convergence theorem, we get \eqref{e:TOLDCT}.
This shows that 
$$\norm{H(u,u_0)(t)-H(u,u_0)(t_0)}_{E^{\delta}}\to 0\ \text{ as }\ t\to t_0^{+}.$$

We will now show that $I_t$ in \eqref{e:IT} tends to zero as $t\to t_0^{-}$.
Given $\eps>0$, we choose $0<\eta\leq\frac{t_0}{2}$ so small that 
\begin{equation}\label{e:CHOICEETA}
\max_{i\in\cali}M(\beta_i,\delta,T)\tfrac{\eta^{\nu_i(\delta)}}{\nu_i(\delta)}
\Bigl(1+\|u\|_{\alpha_i,T}^{\rho_i}\bigl(\tfrac{t_0}{2}\bigr)^{\mu_i-1}\Bigr)\leq\tfrac{\eps}{2c_0|\cali|}.
\end{equation}
Let $t$ be such that $0<t_0-t<\eta$. Then we have $t\in(\frac{t_0}{2},t_0)$ and
\begin{equation*}
\begin{split}
I_t
&\leq\sum_{i\in\cali}\int_{t}^{t_0}\norm{S(t_0-s)\calf_i(u(s))}_{E^{\delta}}ds+
\int_{0}^{t_0-\eta}h_t(s)ds\\
&+\sum_{i\in\cali}\int_{t_0-\eta}^{t}\bigl(\norm{S(t-s)\calf_i(u(s))}_{E^{\delta}}+\norm{S(t_0-s)\calf_i(u(s))}_{E^{\delta}}\bigr)ds,
\end{split}
\end{equation*}
where
$$h_t(s):=\Bigl\|\sum_{i\in\cali}(S(t-s)-S(t_0-s))\calf_i(u(s))\Bigr\|_{E^{\delta}}\ \text{ for }\ s\in(0,t_0-\eta).$$
By \eqref{e:INTEGRALBASIC0} we know that 
\begin{equation*}
\begin{split}
&\sum_{i\in\cali}\int_{t}^{t_0}\norm{S(t_0-s)\calf_i(u(s))}_{E^{\delta}}ds\leq 
c_0\sum_{i\in\cali}M(\beta_i,\delta,T)\tfrac{(t_0-t)^{\nu_i(\delta)}}{\nu_i(\delta)}
\Bigl(1+\|u\|_{\alpha_i,T}^{\rho_i}\bigl(\tfrac{t_0}{2}\bigr)^{\mu_i-1}\Bigr)
\end{split}
\end{equation*}
tends to zero as $t\to t_0^{-}$, and by \eqref{e:CHOICEETA} we have
\begin{equation*}
\begin{split}
\sum_{i\in\cali}&\int_{t_0-\eta}^{t}\bigl(\norm{S(t-s)\calf_i(u(s))}_{E^{\delta}}+\norm{S(t_0-s)\calf_i(u(s))}_{E^{\delta}}\bigr)ds\\ 
&\leq 2c_0\sum_{i\in\cali}M(\beta_i,\delta,T)\tfrac{(t-t_0+\eta)^{\nu_i(\delta)}}{\nu_i(\delta)}
\Bigl(1+\|u\|_{\alpha_i,T}^{\rho_i}(t_0-\eta)^{\mu_i-1}\Bigr)\\
&\leq  
2c_0\sum_{i\in\cali}M(\beta_i,\delta,T)\tfrac{\eta^{\nu_i(\delta)}}{\nu_i(\delta)}
\Bigl(1+\|u\|_{\alpha_i,T}^{\rho_i}\bigl(\tfrac{t_0}{2}\bigr)^{\mu_i-1}\Bigr)\leq\eps.
\end{split}	
\end{equation*}  
Finally, using the same argument as for $g_t$ in \eqref{e:TOLDCT}, by Lebesgue's dominated convergence theorem we get
\begin{equation*}
\int_{0}^{t_0-\eta}h_t(s)ds\to 0\ \text{ as }\ t\to t_0^{-}.
\end{equation*}
This shows that 
$$\norm{H(u,u_0)(t)-H(u,u_0)(t_0)}_{E^{\delta}}\to 0\ \text{ as }\ t\to t_0^{-},$$
which ends the proof of the claim.
\end{proof}  

Next, we formulate additional properties of $\gamma$-solutions.  

\begin{prop}\label{prop:PROPERTIES}
Let \eqref{e:EPSREG2} hold with $\beta_i\in\J_0$, $\alpha_i\in\J$ and $\rho_i\geq 1$ for $i\in\cali$ and assume that
$\gamma\in\J_0$ satisfies \eqref{e:GAMMAALPHAI} and \eqref{e:POSITIVEMUI}. 
Let $u$ be a $\gamma$-solution of \eqref{e:VCF} corresponding to $u_0\in E^\gamma$ defined in $\calt_{\tau}$ for some $\tau\in(0,\infty]$.
If $\delta\in\J$ satisfies \eqref{e:THETA}, then
\begin{itemize}
\item[(i)] we have
\begin{equation}\label{e:CONVGAMMATHETA}
t^{\reg(\delta)-\reg(\gamma)}\norm{u(t)}_{E^{\delta}}\to 0\ \text{ as }t\to 0^{+},
\end{equation}
provided that  \eqref{e:POSITIVEOMEGA} holds and
\begin{equation}\label{e:CONVS}
t^{\reg(\delta)-\reg(\gamma)}\norm{S(t)u_0}_{E^{\delta}}\to 0\ \text{ as }t\to 0^{+},
\end{equation}
\item[(ii)] we have
$$\lim_{t\to 0^{+}}\norm{u(t)-u_0}_{E^{\delta}}=0,$$
provided that  $\reg(\delta)-\reg(\gamma)<\min\limits_{i\in\cali}\omega_i$, $u_0\in E^{\delta}$ and 
$\lim\limits_{t\to 0^+}\|S(t)u_0-u_0\|_{ E^{\delta}}=0$,
\item[(iii)] $u\in C(\calt_{\tau};E^\delta)$ provided that $\{S(t)\}_{t>0}$ is a semigroup on $E^{\gamma}$ and on each $E^{\beta_i}$, $i\in\cali$, and $\delta\in\J$ satisfies also \eqref{e:TODELTA1}.
\end{itemize}
\end{prop}

\begin{proof}
\textit{(i)} Let $T\in\calt_\tau$. By \eqref{e:POSITIVEMUI}, \eqref{e:THETA} and \eqref{e:INTEGRAL1BASIC0} we have for $t\in(0,T]$
\begin{equation}\label{e:CONV1}
\begin{split}
t^{\reg(\delta)-\reg(\gamma)}\norm{u(t)}_{E^{\delta}}&\leq t^{\reg(\delta)-\reg(\gamma)}\norm{S(t)u_0}_{E^{\delta}}\\
&+c_0\sum_{i\in\cali}M(\beta_i,\delta,T)\Bigl(\tfrac{t^{1+\reg(\beta_i)-\reg(\gamma)}}{\nu_i(\delta)}+\norm{u}_{\alpha_i,T}^{\rho_i}t^{\omega_i}B(\nu_i(\delta),\mu_i)\Bigr).
\end{split}
\end{equation}
Since $0<\omega_i\leq 1+\reg(\beta_i)-\reg(\gamma)$ for $i\in\cali$ by \eqref{e:POSITIVEOMEGA}, we obtain \eqref{e:CONVGAMMATHETA} due to \eqref{e:CONVS}.

\noindent
\textit{(ii)} Let $T\in\calt_\tau$. If $\reg(\delta)-\reg(\gamma)<\min\limits_{i\in\cali}\omega_i$ then again by \eqref{e:INTEGRAL1BASIC0} we get for $t\in(0,T]$
\begin{equation*}
\begin{split}
\|u(t)-u_0\|_{E^{\delta}}&\leq\|S(t) u_0-u_0\|_{E^{\delta}}\\
&+c_0\sum_{i\in\cali}M(\beta_i,\delta,T)\Bigl(\tfrac{t^{\nu_i(\delta)}}{\nu_i(\delta)}+\norm{u}_{\alpha_i,T}^{\rho_i}t^{\omega_i+\reg(\gamma)-\reg(\delta)}B(\nu_i(\delta),\mu_i)\Bigr),	
\end{split}
\end{equation*} 
which vanishes as $t\to 0^+$ provided that $S(t)u_0$ tends to $u_0\in E^\delta$ in $E^{\delta}$ as $t\to0^+$.

\noindent
\textit{(iii)} It is a direct consequence of Lemma~\ref{lem:REGH2}.
\end{proof}

From Proposition~\ref{prop:PROPERTIES} we obtain the following particular results.

\begin{cor}\label{cor:CONTINUITY}
Let \eqref{e:EPSREG2} hold with $\beta_i\in\J_0$, $\alpha_i\in\J$ and $\rho_i\geq 1$ for $i\in\cali$
and assume that $\gamma\in\J_0$ satisfies \eqref{e:GAMMAALPHAI} and \eqref{e:POSITIVEMUI}. 
Assume also that $\{S(t)\}_{t>0}$ is a~semigroup on $E^{\gamma}$ and on each $E^{\beta_i}$, $i\in\cali$. Let $u$ be a $\gamma$-solution of \eqref{e:VCF} corresponding to $u_0\in E^\gamma$ defined in $\calt_{\tau}$ for some $\tau\in(0,\infty]$.
\begin{itemize}
\item[(a)] If \eqref{e:SETUP1}, \eqref{e:SETUP2} and \eqref{e:CONDALPHAJ} hold, then
$u\in\bigcap\limits_{i\in\cali}C(\calt_\tau;E^{\alpha_i})$.
\item[(b)] If \eqref{e:TOGAMMA} holds and $\dif(\beta_i,\gamma)<1$ for all $i\in\cali$, then $u\in C(\calt_\tau;E^{\gamma})$. Moreover, if additionally \eqref{e:STC0} and \eqref{e:POSITIVEOMEGA} hold,
then $u$ extends continuously to a function in $C(\calt_\tau\cup\{0\};E^\gamma)$ with $u(0)=u_0$.
\end{itemize}
\end{cor}

Observe that for $\gamma\leadsto\delta$ the boundedness of $t\mapsto t^{\dif(\gamma,\delta)}\norm{S(t)u_0}_{E^\gamma}$ near zero can be improved to \eqref{e:CONVS} if we find $\sigma\in\J_0$ with better regularity index than $\reg(\gamma)$ such that $\sigma\leadsto\delta$ and $E^\sigma\cap E^\gamma$ is dense in $E^\gamma$.   

\begin{prop}\label{prop:CONV2}
Assume that $\gamma\in\J_0$, $\delta\in\J$ are such that $\gamma\leadsto\delta$.  If $\sigma\in\J_0$ is such that $\sigma\leadsto\delta$, $\reg(\gamma)<\reg(\sigma)$ and $E^\sigma\cap E^\gamma$ is dense in $E^\gamma$, then
\begin{equation*}
\lim_{t\to 0^{+}}t^{\dif(\gamma,\delta)}\norm{S(t)u_0}_{E^{\delta}}=0\ \text{ for any }u_0\in E^\gamma.
\end{equation*}
\end{prop} 

\begin{proof}
Since $\sigma\leadsto\delta$, by \eqref{e:LINEAREST} we have for $v_0\in E^\sigma$ and $t\in(0,T]$
$$t^{\dif(\gamma,\delta)}\|S(t) v_0\|_{E^{\delta}}=t^{\reg(\sigma)-\reg(\gamma)}t^{\dif(\sigma,\delta)}\|S(t)v_0\|_{E^{\delta}}\leq M(\sigma,\delta,T)t^{\reg(\sigma)-\reg(\gamma)}\|v_0\|_{E^{\sigma}}\to 0\text{ as }t\to0^{+}.$$
Let $u_0\in E^{\gamma}$, $\eta>0$ and take $v_0\in E^\sigma\cap E^\gamma$ such that $\|u_0-v_0\|_{E^{\gamma}}<\frac{\eta}{2M(\gamma,\delta,T)}$
by  density of $E^\sigma\cap E^\gamma$ in $E^\gamma$.
We choose $0<t_\eta\leq T$ such that
$t^{\dif(\gamma,\delta)}\|S(t)v_0\|_{E^{\delta}}<\frac{\eta}{2}$ for $t\in(0,t_\eta)$.
Then we get by \eqref{e:LINEAREST}
$$t^{\dif(\gamma,\delta)}\|S(t)u_0\|_{E^{\delta}}\leq t^{\dif(\gamma,\delta)}\|S(t)(u_0-v_0)\|_{E^{\delta}}+t^{\dif(\gamma,\delta)}\|S(t)v_0\|_{E^{\delta}}<\eta,\ t\in(0,t_\eta),$$
which proves the claim.
\end{proof}	

We now make an observation that a shift of a $\gamma$-solution is again a $\gamma$-solution, which is continuous at time zero. 

\begin{lem}\label{lem:SHIFT}
Let \eqref{e:EPSREG2} hold with $\rho_i\geq 1$, $\beta_i\in\J_0$ and $\alpha_i\in\J$ for $i\in\cali$
satisfying \eqref{e:SETUP1} and \eqref{e:CONDALPHAJ}
and assume that $\gamma\in\J_0$ satisfies \eqref{e:SETUP2}, \eqref{e:TOGAMMA} and \eqref{e:POSITIVEMUI}. 
Assume also that $\{S(t)\}_{t>0}$ is a~semigroup on $E^{\gamma}$ and on each $E^{\beta_i}$, $i\in\cali$.
Let $u$ be a $\gamma$-solution to \eqref{e:VCF}  corresponding to $u_0\in E^\gamma$ defined in $\calt_{\tau}$ for some $\tau\in(0,\infty]$. 
If $0<\theta<\tau$, then $v(\cdot)=u(\cdot+\theta)\in C(\calt_{\tau-\theta}\cup\{0\};E^\gamma\cap\bigcap\limits_{i\in\cali}E^{\alpha_i})$ is a~$\gamma$-solution of \eqref{e:VCF} on $\calt_{\tau-\theta}$ with $u(\theta)$ in the role of $u_0$, i.e.,
\begin{equation*}
v(t)=S(t)u(\theta)+\sum_{i\in\cali}\int_{0}^{t}S(t-s)\calf_i(v(s))ds\ \text{ for }t\in\calt_{\tau-\theta}.
\end{equation*}  
In particular, $v(t)\to u(\theta)$ in $E^\gamma\cap\bigcap\limits_{i\in\cali}E^{\alpha_i}$ as $t\to0^{+}$.
\end{lem}

\begin{proof}
Note that by \eqref{e:SETUP2} and \eqref{e:CONDALPHAJ} we have 
$$1-\dif(\beta_i,\gamma)=1-\dif(\beta_i,\alpha_i)+\dif(\gamma,\alpha_i)>0\ \text{ for all }i\in\cali.$$ 
Therefore, by Corollary~\ref{cor:CONTINUITY}~(a),~(b) we know that $v\in C(\calt_{\tau-\theta}\cup\{0\};E^\gamma\cap\bigcap\limits_{i\in\cali}E^{\alpha_i})$.
In particular, for every $T\in\calt_{\tau}$ we have $v\in\bigcap\limits_{i\in\cali}{\mathcal{L}}^{\infty}_{\dif(\gamma,\alpha_i)}((0,T-\theta];E^{\alpha_i})$
and for $t\in\calt_{\tau-\theta}$
\begin{equation*}
\begin{split}
v(t)=u(t+\theta)&=S(t+\theta)u_0+\sum_{i\in\cali}\int_{0}^{t+\theta}S(t+\theta-s)\calf_{i}(u(s))ds\\
&=S(t)u(\theta)+\sum_{i\in\cali}\int_{0}^{t}S(t-s)\calf_i(v(s))ds,
\end{split}
\end{equation*}
where we used the fact that $S(t)\in\mathcal{L}(E^\gamma;E^\gamma)$ for $t>0$.
Thus $u(\cdot+\theta)$ is a $\gamma$-solution of \eqref{e:VCF} on $\calt_{\tau-\theta}$ with $u(\theta)$ in the role of $u_0$.
\end{proof}

\section{Existence and uniqueness of $\gamma$-solutions}\label{sec:EXIST}

Following the approach of \cite{Q15,CHQRB17}, in the vein of \cite[Theorem 2.2]{ACRB1999} we prove the local existence and uniqueness of $\gamma$-solutions of \eqref{e:VCF} announced in Theorem~\ref{thm:MAIN1}.
Given the initial condition $u_0\in E^\gamma$, the proof will be based on the fixed point argument for the map $H(\cdot,u_0)$ from \eqref{e:DEFH} in a suitable closed ball
$K(\tau,\mu)=\{u\in K(\tau)\colon \norm{u}_{K(\tau)}\leq\mu\}$
of the Banach space $K(\tau)=\bigcap\limits_{i\in\cali}\mathcal{L}^{\infty}_{\dif(\gamma,\alpha_i)}((0,\tau];E^{\alpha_i})$
endowed with the norm
$$\norm{u}_{K(\tau)}=\max_{i\in\cali}\|u\|_{\alpha_i,\tau}=\max_{i\in\cali}\sup_{t\in(0,\tau]}t^{\dif(\gamma,\alpha_i)}\|u(t)\|_{E^{\alpha_i}}\ \text{ for }u\in K(\tau).$$

\begin{thm}\label{thm:EXIST}
Let $\calf_{i}$ satisfy \eqref{e:EPSREG}, \eqref{e:EPSREG2} with $c_0>0$, $\rho_i\geq 1$, $\alpha_i\in\J$, $\beta_i\in\J_0$ for $i\in\cali$ such that \eqref{e:SETUP1} and \eqref{e:CONDALPHAJ} hold. Assume that $\gamma\in\J_0$ satisfies \eqref{e:SETUP2}, \eqref{e:POSITIVEMUI} and \eqref{e:POSITIVEOMEGA} with $\mu_i>0$ and $\omega_i>0$ for all $i\in\cali$.
Then, given $R>0$, $t_0>0$ and
\begin{equation}\label{e:MU}
\mu\geq 3R\max_{j\in\cali}M(\gamma,\alpha_j,t_0)>0,
\end{equation} 
there exists $\tau\in(0,t_0]$ such that for any $u_0\in E^\gamma$ with $\norm{u_0}_{E^\gamma}\leq R$ there exists a $\gamma$-solution $u$ to \eqref{e:VCF} on $(0,\tau]$, which belongs to $K(\tau,\mu)$. The solution is unique in $K(\tau,\mu)$ and locally unique in time. 
Moreover, there exists $\kappa\in(0,1)$ such that if $\norm{u_0}_{E^\gamma}\leq R$, $\norm{v_0}_{E^\gamma}\leq R$ then we have
\begin{equation}\label{e:LIPINK}
\norm{u-v}_{K(\tau)}\leq\tfrac{1}{1-\kappa}\max_{i\in\cali}M(\gamma,\alpha_i,\tau)\norm{u_0-v_0}_{E^{\gamma}},
\end{equation}
where $u,v\in K(\tau,\mu)$ are $\gamma$-solutions on $(0,\tau]$ corresponding to $u_0,v_0$, respectively. We also have
\begin{equation}\label{e:CONVGAMMATHETANEW}
\norm{u}_{K(t)}\to 0\ \text{ as }t\to 0^{+},
\end{equation}
provided that  
\begin{equation}\label{e:CONVSNEW}
\norm{S(\cdot)u_0}_{K(t)}\to 0\ \text{ as }t\to 0^{+}.
\end{equation}
Furthermore, for $\delta\in\J$ satisfying \eqref{e:TODELTA1} and \eqref{e:THETA} with $\nu_i(\delta)>0$ for all $i\in\cali$, 
\begin{itemize}
\item[(a)] we have
\begin{equation}\label{e:LIPSCHITZ}
t^{\dif(\gamma,\delta)}\norm{u(t)-v(t)}_{E^{\delta}}\leq C_{\delta,\tau}\norm{u_0-v_0}_{E^{\gamma}},\ t\in(0,\tau],
\end{equation}
with 
$$C_{\delta,\tau}=M(\gamma,\delta,\tau)+\tfrac{c_0}{1-\kappa}\max_{j\in\cali}M(\gamma,\alpha_j,\tau)\sum_{i\in\cali}M(\beta_i,\delta,\tau)\kappa_i(\tau,\delta),$$
where
\begin{equation}\label{e:KAPPAI}
\kappa_i(\tau,\delta):=[\tau^{1-\dif(\beta_i,\alpha_i)}+(\|u\|_{\alpha_i,\tau}^{\rho_i-1}+\|v\|_{\alpha_i,\tau}^{\rho_i-1})
\tau^{\omega_i}]B(\nu_i(\delta),\mu_i),
\end{equation}
\item[(b)] we have
\begin{equation}\label{e:CONVGAMMATHETA3}
t^{\dif(\gamma,\delta)}\norm{u(t)}_{E^{\delta}}\to 0\ \text{ as }t\to 0^{+},
\end{equation}
provided that  
\begin{equation}\label{e:CONVS2}
t^{\dif(\gamma,\delta)}\norm{S(t)u_0}_{E^{\delta}}\to 0\ \text{ as }t\to 0^{+}.
\end{equation}
\end{itemize}
If, additionally, $\{S(t)\}_{t>0}$ is a semigroup on $E^{\gamma}$ and on each $E^{\beta_i}$, $i\in\cali$, then
\begin{itemize}
\item[(i)] $u\in\bigcap\limits_{i\in\cali}C((0,\tau];E^{\alpha_i})$,
\item[(ii)] $u\in C((0,\tau];E^\delta)$ provided that $\delta\in\J$ satisfies \eqref{e:TODELTA1} and \eqref{e:THETA},
\item[(iii)] $u\in C((0,\tau];E^\gamma)$ provided that \eqref{e:TOGAMMA} holds,
\item[(iv)] $u\in C([0,\tau];E^\gamma)$ with $u(0)=u_0$ provided that \eqref{e:TOGAMMA} and \eqref{e:STC0} hold.
\end{itemize}
\end{thm}

\begin{proof} 
\textit{Step 1. Complete metric space.}   
Fix $R>0$, $t_0>0$ and, recalling \eqref{e:SETUP2}, let $\mu>0$ satisfy \eqref{e:MU}. By \eqref{e:CONDALPHAJ} we have $\nu_i(\alpha_j)=1-\dif(\beta_i,\alpha_j)>0$ for $i,j\in\cali$.
Since $\mu_i, \omega_i>0$ for all $i\in\cali$ by \eqref{e:POSITIVEMUI} and \eqref{e:POSITIVEOMEGA}, recalling \eqref{e:SETUP1}, we choose $\tau\in(0,t_0]$ so small that 
\begin{equation}\label{e:TAU1}
3\abs{\cali}c_0\max_{i,j\in\cali}M(\beta_i,\alpha_j,\tau)\mu^{\rho_i-1}\tau^{\omega_i}B(\nu_i(\alpha_j),\mu_i)\leq1,
\end{equation}
\begin{equation}\label{e:TAU0B}
3\abs{\cali}c_0\max_{i,j\in\cali}\tau^{\nu_i(\alpha_j)+\dif(\gamma,\alpha_j)}\tfrac{M(\beta_i,\alpha_j,\tau)}{\nu_i(\alpha_j)}\leq\mu
\end{equation}
and
\begin{equation}\label{e:TAUCONTR}
3\abs{\cali}c_0\max_{i,j\in\cali}M(\beta_i,\alpha_j,\tau)\tau^{1-\dif(\beta_i,\alpha_i)}B(\nu_i(\alpha_j),\mu_i)<1.
\end{equation}
Using this $\mu$, we define the closed ball $K(\tau,\mu)$ in the Banach space $K(\tau)$.

\noindent 
\textit{Step 2. Transformation and its fixed point.}  
Let $u_0\in E^\gamma$ be such that $\|u_0\|_{E^\gamma}\leq R$. 
Recalling \eqref{e:DEFH}, we consider
\begin{equation}\label{e:MAP}
H(u,u_0)(t)=S(t) u_0+\sum_{i\in\cali}\int_{0}^{t}S(t-s)\calf_i(u(s))ds,\ t\in(0,\tau],\ u\in K(\tau,\mu).
\end{equation}
Note that $H(\cdot,u_0)\colon K(\tau,\mu)\to K(\tau,\mu)$ is well defined. Indeed, by Lemma~\ref{lem:REGH}~(ii) for $u\in K(\tau,\mu)$ we know that $H(u,u_0)\in K(\tau)$ and, by \eqref{e:SETUP2} and \eqref{e:OSZHINL} with $\delta=\alpha_j$, we get
\begin{equation*}
\begin{split}
t^{\dif(\gamma,\alpha_j)}\norm{H(u,u_0)(t)}_{E^{\alpha_j}}&\leq M(\gamma,\alpha_j,t_0)R+c_0\sum_{i\in\cali}\tau^{\nu_i(\alpha_j)+\dif(\gamma,\alpha_j)}\tfrac{M(\beta_i,\alpha_j,\tau)}{\nu_i(\alpha_j)}\\
&+c_0\sum_{i\in\cali}M(\beta_i,\alpha_j,\tau)\mu^{\rho_i}\tau^{\omega_i}B(\nu_i(\alpha_j),\mu_i),\ t\in(0,\tau],\ j\in\cali.
\end{split}
\end{equation*}
Then by the choice of $\mu$ and $\tau$ in \eqref{e:MU}, \eqref{e:TAU1} and \eqref{e:TAU0B} we obtain $\norm{H(u,u_0)}_{K(\tau)}\leq\mu$.

The map $H(\cdot,u_0)$ is a strict contraction on $K(\tau,\mu)$. Indeed, using Lemma~\ref{lem:BASIC1} with $\delta=\alpha_j$, $j\in\cali$, we get
\begin{equation*}
\norm{H(u,u_0)-H(v,u_0)}_{K(\tau)}\leq\kappa\norm{u-v}_{K(\tau)},\ u,v\in  K(\tau,\mu),  
\end{equation*}
with
\begin{equation}\label{e:KAPPA}
\begin{split}
\kappa=c_0\sum_{i\in\cali}\max_{j\in\cali}M(\beta_i,\alpha_j,\tau)&\bigl[\tau^{1-\dif(\beta_i,\alpha_i)}+2\mu^{\rho_i-1}\tau^{\omega_i}\bigr]B(\nu_i(\alpha_j),\mu_i).
\end{split}
\end{equation}
By \eqref{e:TAU1} and \eqref{e:TAUCONTR} we know that $\kappa<1$. Therefore, by the Banach fixed point theorem there exists a unique fixed point $u\in K(\tau,\mu)$ of $H(\cdot,u_0)$, i.e., 
\begin{equation*}
u(t)=H(u,u_0)(t)=S(t)u_0+\sum_{i\in\cali}\int_{0}^{t}S(t-s)\calf_i(u(s))ds,\ t\in (0,\tau].
\end{equation*}
Thus $u$ is a $\gamma$-solution of $\eqref{e:VCF}$ on $(0,\tau]$, which is unique in $K(\tau,\mu)$.

\noindent
\textit{Step 3. Local uniqueness in time.} Suppose that $u_1,u_2$ are two $\gamma$-solutions of \eqref{e:VCF} with $u_0\in E^{\gamma}$ on $(0,\tau]$. Taking in Step 1 $t_0=\tau$, $R=\norm{u_0}_{E^{\gamma}}$ and     
$$\mu=\bar\mu=\max\{3\max_{j\in\cali}M(\gamma,\alpha_j,\tau)\|u_0\|_{E^\gamma}, \max_{k=1,2}\max_{i\in\cali}\norm{u_k}_{\alpha_i,\tau}\},$$
we choose $0<\bar\tau\leq\tau$ so small that \eqref{e:TAU1}--\eqref{e:TAUCONTR} hold with $\bar\tau$ and $\bar{\mu}$ in the roles of $\tau$ and $\mu$.  Since $u_1,u_2\in K(\bar\tau,\bar\mu)$ are $\gamma$-solutions of \eqref{e:VCF} on $(0,\bar\tau]$, they are fixed points in $K(\bar\tau,\bar\mu)$ of $H(\cdot,u_0)$ given in \eqref{e:MAP} and thus must coincide on $(0,\bar\tau]$, i.e., $u_1(t)=u_2(t)$ for $t\in(0,\bar\tau]$.

\noindent 
\textit{Step 4. Lipschitz dependence of $\gamma$-solutions on $u_0$ in bounded subsets of $E^\gamma$.} Let $u_0, v_0\in E^{\gamma}$ be such that $\norm{u_0}_{E^\gamma},\norm{v_0}_{E^\gamma}\leq R$ and let $u,v\in K(\tau,\mu)$ be the corresponding $\gamma$-solutions on $(0,\tau]$ constructed in Steps 1--2.

If $\delta\in\J$ satisfies \eqref{e:TODELTA1} and \eqref{e:THETA}, then again by Lemma~\ref{lem:BASIC1} we have for $t\in(0,\tau]$
\begin{equation*}
\begin{split}
t^{\dif(\gamma,\delta)}\norm{u(t)-v(t)}_{E^{\delta}}&\leq M(\gamma,\delta,\tau)\norm{u_0-v_0}_{E^{\gamma}}+c_0\sum_{i\in\cali}M(\beta_i,\delta,\tau)\kappa_i(\tau,\delta)\norm{u-v}_{K(\tau)}
\end{split}
\end{equation*} 
with $\kappa_i(\tau,\delta)$ given in \eqref{e:KAPPAI}. By \eqref{e:SETUP1}, \eqref{e:CONDALPHAJ} we can take $\delta=\alpha_j$ for $j\in\cali$ and, thanks to \eqref{e:TAU1} and \eqref{e:TAUCONTR}, we obtain
\begin{equation*}
\norm{u-v}_{K(\tau)}\leq\max_{j\in\cali}M(\gamma,\alpha_j,\tau)\norm{u_0-v_0}_{E^{\gamma}}+\kappa\norm{u-v}_{K(\tau)}
\end{equation*} 
with $\kappa\in(0,1)$ given in \eqref{e:KAPPA}. This yields \eqref{e:LIPINK} and in consequence \eqref{e:LIPSCHITZ}. 

\noindent 
\textit{Step 5. Profile at $t=0$.} Note that \eqref{e:CONVS2} implies \eqref{e:CONVGAMMATHETA3} by Proposition~\ref{prop:PROPERTIES}. In fact, taking $\delta=\alpha_j$ for $j\in\cali$ in \eqref{e:CONV1} we obtain \eqref{e:CONVGAMMATHETANEW} if \eqref{e:CONVSNEW} holds.

\noindent 
\textit{Step 6. Continuity of $u$.}
Assume now that $\{S(t)\}_{t>0}$ is a semigroup on $E^{\gamma}$ and on each $E^{\beta_i}$ for $i\in\cali$.
Then $(ii)$ follows directly from Proposition~\ref{prop:PROPERTIES}, whereas $(i)$, $(iii)$ and $(iv)$ are consequences of Corollary~\ref{cor:CONTINUITY}, since $\dif(\beta_i,\gamma)=\dif(\beta_i,\alpha_i)-\dif(\gamma,\alpha_i)<1$  for $i\in\cali$ 
by \eqref{e:SETUP2} and \eqref{e:CONDALPHAJ}.
\end{proof}

Under the additional assumption of Theorem~\ref{thm:EXIST} the $\gamma$-solution is in fact unique there.

\begin{thm}\label{thm:UNIQUENESS}
Let \eqref{e:EPSREG}, \eqref{e:EPSREG2}, \eqref{e:SETUP1} and \eqref{e:CONDALPHAJ} hold and let $\gamma\in\J_0$ satisfy \eqref{e:SETUP2} and \eqref{e:POSITIVEMUI}. 
Assume that $\{S(t)\}_{t>0}$ is a semigroup on $E^{\gamma}$ and on each $E^{\beta_i}$ for $i\in\cali$.
If $u_1,u_2$ are two $\gamma$-solutions of \eqref{e:VCF} with $u_0\in E^{\gamma}$ on $(0,T]$ with some $T\in(0,\infty)$ 
and they coincide on $(0,\bar\tau]$ for some $\bar{\tau}\in(0,T)$, i.e., $u_1(t)=u_2(t)$ for $t\in(0,\bar\tau]$,
then they coincide on $(0,T]$.	
\end{thm}

\begin{proof} 
Suppose that $u_1,u_2$ are two $\gamma$-solutions of \eqref{e:VCF} with $u_0\in E^{\gamma}$ on $(0,T]$ with some $T\in(0,\infty)$,
which coincide on $(0,\bar\tau]$ for some $\bar{\tau}\in(0,T)$.
To show that they also coincide on $[\bar\tau,T]$, we define the shifts $\bar{u}_1(t):=u_1(t+\bar\tau)$ and  $\bar{u}_2(t):=u_2(t+\bar\tau)$ for $t\in[0,T-\bar\tau]$. Note that  $\bar{u}_1(0)=u_1(\bar\tau)=u_2(\bar\tau)=\bar{u}_2(0)$ and $\bar{u}_1,\bar{u}_2\in\bigcap\limits_{i\in\cali}C([0,T-\bar\tau];E^{\alpha_i})$ by Corollary~\ref{cor:CONTINUITY}~(a).
Setting $z:=\bar{u}_1-\bar{u}_2$, we have $z(0)=0$ and
\begin{equation*}
z(t)=\sum_{i\in\cali}\int_{0}^{t}S(t-s)\left(\calf_i(\bar{u}_1(s))-\calf_i(\bar{u}_2(s))\right)ds\ \text{ for }\ t\in[0,T-\bar\tau]. 	
\end{equation*}
Using \eqref{e:EPSREG} and \eqref{e:SETUP1}, for $j\in\cali$ and $t\in[0,T-\bar\tau]$ we get
\begin{equation*}
\|z(t)\|_{E^{\alpha_j}}\leq
\sum_{i\in\cali}\int_{0}^{t}\tfrac{c_0M(\beta_i,\alpha_j,T-\bar\tau)}{(t-s)^{\dif(\beta_i,\alpha_j)}}\Bigl(1+\|\bar{u}_1(s)\|_{E^{\alpha_i}}^{\rho_i-1}+\|\bar{u}_2(s)\|^{\rho_i-1}_{E^{\alpha_i}}\Bigr)\|z(s)\|_{E^{\alpha_i}}ds.
\end{equation*}
Since $\bar{u}_1,\bar{u}_2$ are bounded in each $E^{\alpha_i}$ on $[0,T-\bar\tau]$, we obtain 
\begin{equation*}
\|z(t)\|_{E^{\alpha_j}}\leq 
\sum_{i\in\cali}\int_{0}^{t}\tfrac{C_iM(\beta_i,\alpha_j,T-\bar\tau)}{(t-s)^{\dif(\beta_i,\alpha_j)}}\|z(s)\|_{E^{\alpha_i}}ds,\ t\in[0,T-\bar\tau],
\end{equation*}
with some $C_i>0$, $i\in\cali$. Let $\alpha^{*}\in\{\alpha_i\colon i\in\cali\}$, $\beta_{*}\in\{\beta_i\colon i\in\cali\}$ be such that 
$$\reg(\alpha^{*})=\max\{\reg(\alpha_i)\colon i\in\cali\}\ \text{ and }\ \reg(\beta_{*})=\min\{\reg(\beta_i)\colon i\in\cali\}.$$ 
Thus we conclude that
\begin{equation*}	
\sum_{i\in\cali}\|z(t)\|_{E^{\alpha_i}}\leq \abs{\cali}C_0\int_{0}^{t}\tfrac{1}{(t-s)^{\dif(\beta_{*},\alpha^{*})}}\sum_{i\in\cali}\|z(s)\|_{E^{\alpha_i}}ds, \ t\in[0,T-\bar\tau],
\end{equation*}
where $C_0=\max\limits_{i,j\in\cali}C_i M(\beta_i,\alpha_j,T-\bar\tau)(T-\bar\tau)^{\reg(\alpha^{*})-\reg(\alpha_j)}(T-\bar\tau)^{\reg(\beta_i)-\reg(\beta_{*})}$. 
Since $\dif(\beta_{*},\alpha^{*})<1$, by the Volterra type inequality (see \cite[Lemma~1.2.9]{C-D}) we obtain $z(t)=0$ in $[0,T-\bar{\tau}]$, i.e., $u_1=u_2$  on $(0,T]$.
\end{proof}

\begin{rem}
Regarding a relation of \eqref{e:POSITIVEOMEGA} with \eqref{e:CONDALPHAJ} and \eqref{e:POSITIVEMUI}, we observe that, given $i\in\cali$, we have $\omega_i=1-\rho_i\dif(\gamma,\alpha_i)+\reg(\beta_i)-\reg(\gamma)>0$ in each of the following cases:
\begin{itemize}
\item[(i)] $\rho_i>1$, $\rho_i\dif(\gamma,\alpha_i)<1$ and $\rho_i\dif(\beta_i,\alpha_i)\leq 1$,
\item[(ii)] $\rho_i=1$, $\dif(\gamma,\alpha_i)<1$ and $\dif(\beta_i,\alpha_i)<1$.
\end{itemize}
Thus if \eqref{e:CONDALPHAJ} and \eqref{e:POSITIVEMUI} hold then some $\omega_i$ can be zero only if $\rho_i>1$ and $\rho_i\dif(\beta_i,\alpha_i)>1$.
\end{rem}

The existence of a local $\gamma$-solution can be also proved if we replace assumption \eqref{e:POSITIVEOMEGA} in Theorem~\ref{thm:EXIST} by its weaker counterpart \eqref{e:NONNEGATIVEOMEGA}, hence admitting some of $\omega_i$'s to be zero.

\begin{thm}\label{thm:EXISTCRITICAL}
Under assumptions \eqref{e:EPSREG}, \eqref{e:EPSREG2}, \eqref{e:SETUP1} and \eqref{e:CONDALPHAJ} with $c_0>0$, $\rho_i>1$, $\alpha_i\in\J$ and $\beta_i\in\J_0$ for $i\in\cali$, assume that $\gamma\in\J_0$ satisfies \eqref{e:SETUP2}, \eqref{e:POSITIVEMUI} and \eqref{e:NONNEGATIVEOMEGA} with $\mu_i>0$ and $\omega_i\geq0$ for all $i\in\cali$.
Let $w_0\in E^\gamma$ be such that
\begin{equation}\label{e:SMALLW}
\max_{j\in\cali}t^{\dif(\gamma,\alpha_j)}\norm{S(t)w_0}_{E^{\alpha_j}}\to 0\text{ as }t\to 0^{+}.
\end{equation}
Then there exist $r,\tau,\mu\in(0,\infty)$ such that for any $u_0\in E^\gamma$ with $\norm{u_0-w_0}_{E^\gamma}\leq r$ there exists a $\gamma$-solution of \eqref{e:VCF} on $(0,\tau]$, which is unique in $K(\tau,\mu)$. We also have \eqref{e:CONVUTOZERO}.
If $\{S(t)\}_{t>0}$ is a semigroup on $E^{\gamma}$ and on each $E^{\beta_i}$ for $i\in\cali$, then $u$ satisfies assertions $(i)$--$(iv)$ from Theorem~\ref{thm:EXIST} and $u$ is unique if $u_0\in E^\gamma$ satisfies \eqref{e:CONVSNEW}.
\end{thm} 

\begin{proof}
Fix $t_0>0$ and choose $\mu>0$ so small that
\begin{equation}\label{e:NEWMU}
3|\cali|c_0\max_{i,j\in\cali}M(\beta_i,\alpha_j,t_0)\mu^{\rho_i-1}t_0^{\omega_i}B(\nu_i(\alpha_j),\mu_i)\leq 1.
\end{equation} 
Using \eqref{e:SMALLW} we choose $\tau\in(0,t_0]$ so small that \eqref{e:TAU0B}, \eqref{e:TAUCONTR} and
\begin{equation}\label{e:TAUW0}
6\max_{j\in\cali}\sup_{t\in(0,\tau]}t^{\dif(\gamma,\alpha_j)}\norm{S(t)w_0}_{E^{\alpha_j}}\leq\mu
\end{equation}
hold. Finally, let $r>0$ be so small that 
\begin{equation}\label{e:DEFR}
6\max_{j\in\cali}M(\gamma,\alpha_j,t_0)r\leq\mu.
\end{equation}
Recalling \eqref{e:MAP}, we have $H(\cdot,u_0)\in K(\tau,\mu)$  for $u_0\in K(\tau,\mu)$, since by Lemma~\ref{lem:REGH} we get for $t\in(0,\tau]$
and $j\in\cali$
\begin{equation*}
\begin{split}
t^{\dif(\gamma,\alpha_j)}\norm{H(u,u_0)(t)}_{E^{\alpha_j}}&\leq M(\gamma,\alpha_j,t_0)\norm{u_0-w_0}_{E^\gamma}+t^{\dif(\gamma,\alpha_j)}\norm{S(t)w_0}_{E^{\alpha_j}}\\
&+c_0\sum_{i\in\cali}M(\beta_i,\alpha_j,\tau)\Bigl[\tfrac{\tau^{\nu_i(\alpha_j)+\dif(\gamma,\alpha_j)}}{\nu_i(\alpha_j)}+\mu^{\rho_i}\tau^{\omega_i}B(\nu_i(\alpha_j),\mu_i)\Bigr]
\end{split}
\end{equation*}
and  \eqref{e:TAU0B}, \eqref{e:NEWMU}, \eqref{e:TAUW0}, and \eqref{e:DEFR} yield the desired estimate. Moreover, due to \eqref{e:TAUCONTR},  $H(\cdot,u_0)$ is a $\kappa$-contraction on $K(\tau,\mu)$ with $\kappa<1$ given in \eqref{e:KAPPA}. For the fixed point $u$ of $H(\cdot,u_0)$ by \eqref{e:OSZHINL} we have for $0<t\leq\tau$
\begin{equation*}
\begin{split}
\|u\|_{K(t)}\leq\|S(\cdot)u_0\|_{K(t)}&+c_0\sum_{i\in\cali}\max_{j\in\cali}t^{\nu_i(\alpha_j)+\dif(\gamma,\alpha_j)}\tfrac{M(\beta_i,\alpha_j,t_0)}{\nu_i(\alpha_j)}\\
&+c_0\sum_{i\in\cali}\max_{j\in\cali}M(\beta_i,\alpha_j,t_0)\mu^{\rho_i-1}t_0^{\omega_i}B(\nu_i(\alpha_j),\mu_i)\|u\|_{K(t)},
\end{split}
\end{equation*}
which by \eqref{e:NEWMU} gives
$$\tfrac{2}{3}\|u\|_{K(t)}\leq\|S(\cdot)u_0\|_{K(t)}+c_0\sum_{i\in\cali}\max_{j\in\cali}t^{\nu_i(\alpha_j)+\dif(\gamma,\alpha_j)}\tfrac{M(\beta_i,\alpha_j,t_0)}{\nu_i(\alpha_j)}.$$
Thus we have \eqref{e:CONVUTOZERO}. 

If $\{S(t)\}_{t>0}$ is a semigroup on $E^{\gamma}$ and on each $E^{\beta_i}$ for $i\in\cali$, Step 6 of the proof of Theorem~\ref{thm:EXIST} carries over in the present situation. Moreover, if $u_1, u_2$ are two $\gamma$-solutions of \eqref{e:VCF} on $(0,\tau]$ corresponding to $u_0\in E^\gamma$, which satisfies $\|u_0-w_0\|_{E^\gamma}\leq r$ and \eqref{e:CONVSNEW}, then $\|u_k\|_{K(t)}$ tends to zero as $t\to 0^{+}$ for $k=1,2$ due to \eqref{e:CONVUTOZERO}. Taking $\bar\tau\in(0,\tau]$ small enough, we see that $u_k\in K(\bar\tau,\mu)$, $k=1,2$, are fixed points of $H(\cdot,u_0)$ in $K(\bar\tau,\mu)$. However, the above argument implies that they need to coincide. This shows the local uniqueness, whereas their uniqueness on $(0,\tau]$ follows from Theorem~\ref{thm:UNIQUENESS}.
\end{proof}

In Theorem~\ref{thm:EXIST}, the existence time $\tau$ was found as a common time of existence of $\gamma$-solutions for $u_0$ taken from a given ball in $E^\gamma$. In the corollary below, we find a time $\tau(u_0)>0$ of existence for an individual $\gamma$-solution of \eqref{e:VCF} with $u_0\in E^\gamma$. 

\begin{cor}\label{cor:TAUU0}
Under the assumptions of Theorem~\ref{thm:EXIST}, if $u_0\in E^\gamma$ then in Theorem~\ref{thm:EXIST} we can take
$\tau=\tau(u_0)\in(0,1]$ given by
\begin{equation}\label{e:time}
\tau(u_0):=\Bigl(\tfrac{1}{6\abs{\cali}\lambda(1+\norm{u_0}_{E^\gamma})^{\rho-1}}\Bigr)^{\tfrac{1}{\omega}},
\end{equation}
with $\omega:=\min\limits_{i\in\cali}\omega_i>0$, $\rho:=\max\limits_{i\in\cali}\rho_i\geq 1$ and 
\begin{equation*}
\begin{split}
\lambda:=\max_{i,j\in\cali}\Big\{&c_0M_{ij}(1)[(1+3\max_{k\in\cali}M(\gamma,\alpha_k,1))^{\rho_i}B_{ij}+1],\tfrac{c_0M_{ij}(1)}{\nu_i(\alpha_j)}, c_0M_{ij}(1)B_{ij},\tfrac{1}{6}\Big\},
\end{split}	
\end{equation*}
where
$$M_{ij}(s)=M(\beta_i,\alpha_j,s)\ \text{ and }\ B_{ij}=B(\nu_i(\alpha_j),\mu_i)\ \text{ for }i,j\in\cali,\ s>0$$ 
with $\mu_i=1-\rho_i\dif(\gamma,\alpha_i)$ and $\nu_i(\alpha_j)=1-\dif(\beta_i,\alpha_j)$.
\end{cor}

\begin{proof}
Setting $R=\|u_0\|_{E^\gamma}$, $t_0=1$, $m=\max\limits_{j\in\cali}M(\gamma,\alpha_j,1)$ and $\mu=3mR+R+1$ in Step~1 of the proof of Theorem~\ref{thm:EXIST}, we need to verify \eqref{e:TAU1}, \eqref{e:TAU0B} and \eqref{e:TAUCONTR} with $\tau=\tau(u_0)$.
Recalling $\omega_i$, $i\in\cali$, from \eqref{e:POSITIVEOMEGA}, we have
\begin{equation}\label{e:TAU1BETTER}
\begin{split}
c_0\tau(u_0)^{\omega_i}M_{ij}(\tau(u_0))\mu^{\rho_i}B_{ij}&\leq c_0\tau(u_0)^{\omega}M_{ij}(1)[(1+3m)^{\rho_i}B_{ij}+1](1+R)^{\rho_i}\\
&\leq\tau(u_0)^{\omega}\lambda(1+R)^{\rho}=\tfrac{1+R}{6\abs{\cali}}\leq\tfrac{\mu}{6\abs{\cali}}\ \text{ for }i,j\in\cali.
\end{split}
\end{equation}
Consequently, we obtain \eqref{e:TAU1} with $\tau=\tau(u_0)$.

By \eqref{e:POSITIVEOMEGA} we have $\omega_i=\nu_i(\alpha_j)+\dif(\gamma,\alpha_j)-\rho_i\dif(\gamma,\alpha_i)$ for $i\in\cali$. Thus we get for $i,j\in\cali$ 
\begin{equation*}
\begin{split}
&c_0\tau(u_0)^{\nu_i(\alpha_j)+\dif(\gamma,\alpha_j)}\tfrac{M_{ij}(\tau(u_0))}{\nu_i(\alpha_j)}\leq c_0\tau(u_0)^{{\omega_i}}\tfrac{M_{ij}(1)}{\nu_i(\alpha_j)}\leq\tau(u_0)^{\omega}\lambda(1+R)^{\rho}=\tfrac{1+R}{6\abs{\cali}}\leq\tfrac{\mu}{6\abs{\cali}},
\end{split}
\end{equation*}
which implies \eqref{e:TAU0B} with $\tau=\tau(u_0)$.

Since $\omega_i\leq 1-\dif(\beta_i,\alpha_i)$ for $i\in\cali$, we have 
\begin{equation}\label{e:TAUCONTRBETTER}
6\abs{\cali}c_0M_{ij}(\tau(u_0))\tau(u_0)^{1-\dif(\beta_i,\alpha_i)}B_{ij}
\leq6\abs{\cali}\lambda \tau(u_0)^{\omega}\leq 1\ \text{ for }i,j\in\cali, 
\end{equation}
which yields \eqref{e:TAUCONTR} with $\tau=\tau(u_0)$. Having verified \eqref{e:TAU1}, \eqref{e:TAU0B} and \eqref{e:TAUCONTR} with $\tau=\tau(u_0)$, we conclude that
$H(\cdot,u_0)\colon K(\tau(u_0),\mu)\to K(\tau(u_0),\mu)$ in \eqref{e:MAP} is a well-defined contraction and by \eqref{e:TAU1BETTER} and \eqref{e:TAUCONTRBETTER} we have
\begin{equation*}
\norm{H(u,u_0)-H(v,u_0)}_{K(\tau(u_0))}\leq \frac{1}{2}\norm{u-v}_{K(\tau(u_0))},\ u,v\in  K(\tau(u_0),\mu).
\end{equation*}	
The proof is complete.
\end{proof}

We will now prove Corollary~\ref{cor:EXIST} announced in the Introduction. It states that each $\gamma$-solution of \eqref{e:VCF} with $u_0\in E^\gamma$, constructed via Theorem~\ref{thm:EXIST} and Corollary~\ref{cor:TAUU0}, which belongs to $C((0,\tau(u_0)];E^\gamma)$, has a unique extension to a~$\gamma$-solution $u(\cdot,u_0)$ of \eqref{e:VCF}  defined on the interval $(0,\tau_{u_0})$ with maximal time of existence $\tau_{u_0}>\tau(u_0)$.

\begin{proof}[Proof of Corollary~\ref{cor:EXIST}]
By Theorems~\ref{thm:EXIST},~\ref{thm:UNIQUENESS} and Corollary~\ref{cor:TAUU0} we know that a $\gamma$-solution $u$ of \eqref{e:VCF} with $u_0\in E^\gamma$ exists on $(0,\tau(u_0)]$ and is unique. Suppose that we extend $u$ to a~$\gamma$-solution of \eqref{e:VCF} on some interval $(0,T]$. Denoting the extension again by $u$, due to \eqref{e:TOGAMMA} we have $u\in C((0,T];E^\gamma)$.
We choose $0<\epsilon<T$ and let $R>0$ be such that
$$\sup_{t\in[\epsilon,T]}\|u(t)\|_{E^\gamma}\leq R.$$
By Theorem~\ref{thm:EXIST} we find a common time of existence $\tau>0$ for all $u(s)$, $s\in[\epsilon,T]$, and construct $\gamma$-solutions $v_s$ on $(0,\tau]$ of \eqref{e:VCF} with $u(s)\in E^\gamma$ in the role of $u_0$, that is,
$$v_s(t)=S(t)u(s)+\sum_{i\in\cali}\int_0^t S(t-r)\calf_i(v_{s}(r))dr,\ t\in(0,\tau],$$
and
$$v_s\in \bigcap_{i\in\cali}C((0,\tau];E^\gamma\cap E^{\alpha_i})\cap\mathcal{L}^\infty_{\dif(\gamma,\alpha_i)}((0,\tau];E^{\alpha_i}).$$ 
We take $\theta\in[\epsilon,T)$ such that  $0<T-\theta<\tau$. We consider the shift $u_{\theta}=u(\cdot+\theta)$ which by Lemma~\ref{lem:SHIFT} is a $\gamma$-solution on $[0,T-\theta]$ of \eqref{e:VCF} with $u(\theta)$ in the role of $u_0$ and 
$$u_{\theta}\in C([0,T-\theta];E^\gamma\cap\bigcap_{i\in\cali}E^{\alpha_i}).$$
By Theorem~\ref{thm:UNIQUENESS} we know that $u_{\theta}=v_{\theta}$ on the interval $(0,T-\theta]$. Thus $v_{\theta}$ extends continuously in $E^\gamma\cap\bigcap\limits_{i\in\cali}E^{\alpha_i}$ at time zero by setting $v_{\theta}(0)=u(\theta)$. We have
$$\sup_{t\in[0,\tau]}\|v_{\theta}(t)\|_{E^{\gamma}}+\max_{i\in\cali}\sup_{t\in[0,\tau]}\|v_{\theta}(t)\|_{E^{\alpha_i}}\leq M$$
for some $M>0$. 
We consider the function $z\colon(0,\tau+\theta]\to E^\gamma\cap\bigcap\limits_{i\in\cali}E^{\alpha_i}$ given by 
$$z(t)=u(t)\ \text{ for }t\in(0,\theta]\quad \text{ and }\quad z(t)=v_{\theta}(t-\theta)\ \text{ for }t\in(\theta,\tau+\theta].$$
Note that $z$ belongs to $K(\tau+\theta)$, 
since for $i\in\cali$ we have
\begin{equation*}
\begin{split}
\sup_{t\in(0,\tau+\theta]}t^{\dif(\gamma,\alpha_i)}\|z(t)\|_{E^{\alpha_i}}&\leq\sup_{t\in(0,\theta]}t^{\dif(\gamma,\alpha_i)}\|u(t)\|_{E^{\alpha_i}}+\sup_{t\in[\theta,\tau+\theta]}t^{\dif(\gamma,\alpha_i)}\|v_{\theta}(t-\theta)\|_{E^{\alpha_i}}\\
&\leq\sup_{t\in(0,T]}t^{\dif(\gamma,\alpha_i)}\|u(t)\|_{E^{\alpha_i}}+\sup_{t\in[0,\tau]}(t+\theta)^{\dif(\gamma,\alpha_i)}\|v_{\theta}(t)\|_{E^{\alpha_i}}\\
&\leq\norm{u}_{\alpha_i,T}+(\tau+T)^{\dif(\gamma,\alpha_i)}M<\infty.
\end{split}
\end{equation*}
Moreover, we know that $z\in C((0,\tau+\theta];E^\gamma\cap\bigcap\limits_{i\in\cali}E^{\alpha_i})$ and $z$ satisfies \eqref{e:VCF} with $u_0$ on $(0,\tau+\theta]$, since for $t\in[\theta,\tau+\theta]$ we have 
$$z(t)=S(t-\theta)u(\theta)+\sum_{i\in\cali}\int_{\theta}^{t}S(t-s)\calf_i(z(s))ds=S(t)u_0+\sum_{i\in\cali}\int_{0}^{t}S(t-s)\calf_i(z(s))ds.$$
Hence $z$ is a $\gamma$-solution of \eqref{e:VCF} on $(0,\tau+\theta]$, which extends $u$ beyond $(0,T]$. 

We define the maximal time of existence $\tau_{u_0}\in(0,\infty]$ by
$$\tau_{u_0}:=\sup\{T>0\colon z\colon(0,T]\to E^\gamma\cap\bigcap_{i\in\cali}E^{\alpha_i}\ \text{is a $\gamma$-solution of }\eqref{e:VCF} \text{ which extends }\ u\}.$$ 
Since Theorem~\ref{thm:UNIQUENESS} guarantees that each extension of $u$ is unique, we can define the function
$u(\cdot,u_0)\colon(0,\tau_{u_0})\to E^\gamma\cap\bigcap\limits_{i\in\cali}E^{\alpha_i}$ 
by 
$$u(t,u_0)=z(t),\ t\in(0,T]\text{ for any }T<\tau_{u_0},$$
where $z$ is a $\gamma$-solution of \eqref{e:VCF} on $(0,T]$. 
Note that $u(\cdot,u_0)$ is a $\gamma$-solution of \eqref{e:VCF} with $u_0\in E^\gamma$ in the sense of Definition~\ref{defn:GAMMASOL}
and has regularity as specified in \eqref{e:HOWREGULAR}.

Suppose now that $\tau_{u_0}<\infty$.
By Lemma~\ref{lem:SHIFT} and Theorem~\ref{thm:UNIQUENESS} for any $0<t<\tau_{u_0}$ we see by uniqueness of $\gamma$-solutions that
$u(\cdot+t,u_0)=u(\cdot,u(t,u_0))$ is a $\gamma$-solution with the maximal time of existence $\tau_{u(t,u_0)}=\tau_{u_0}-t$. From Corollary~\ref{cor:TAUU0} we know that $\tau_{u(t,u_0)}>\tau(u(t,u_0))$. Moreover, by \eqref{e:time} we have
$$(1+\norm{u(t,u_0)}_{E^\gamma})^{\rho-1}>\tfrac{1}{6\abs{\cali}\lambda(\tau_{u_0}-t)^\omega}$$
and thus, given $\eps>0$, we get
\begin{equation}\label{e:timelife}
\|u(t,u_0)\|_{E^{\gamma}}>\tfrac{1}{\bigl(6\abs{\cali}\lambda(\tau_{u_0}-t)^{\omega}\bigr)^\frac{1}{\max\{\rho-1,\eps\}}}-1\ \text{ for all }\ 0<t<\tau_{u_0},
\end{equation}
which implies \eqref{e:LIMSUP}.
\end{proof}

The result below shows that if a $\gamma$-solution extended to the maximal time of existence $\tau_{u_0}$ is suitably integrable then it must be global.

\begin{cor}
Assume that $u(\cdot,u_0)$ is a $\gamma$-solution of \eqref{e:VCF} on $(0,\tau_{u_0})$ with the maximal time of existence $\tau_{u_0}$. Let $p\in[1,\infty]$ be such that $p\geq\frac{\rho-1}{\omega}$,
where $\rho=\max\limits_{i\in\cali}\rho_i\geq 1$ and $\omega=\min\limits_{i\in\cali}\omega_i>0$.
If for any finite time interval $[\tau_1,\tau_2)\subset(0,\tau_{u_0})$ we have
\begin{equation}\label{e:INTEGRABILITY}
u(\cdot,u_0)\in L^p ((\tau_1,\tau_2);E^{\gamma}),
\end{equation}  
then $t_{u_0}=\infty$, that is, $u(t,u_0)$ is defined for all $t>0$.
\end{cor}

\begin{proof}
Suppose contrary to the claim that $\tau_{u_0}<\infty$. Take $\eps=\omega>0$ and note that then $p\omega\geq\max\{\rho-1,\eps\}$ and by \eqref{e:timelife} for all $t$ close enough to $\tau_{u_0}$ we have
$$\|u(t,u_0)\|_{E^{\gamma}}>c(\tau_{u_0}-t)^{-\frac{\omega}{\max\{\rho-1,\eps\}}}$$
with some $c>0$. Then with $\tau$ close enough to $\tau_{u_0}$  we have $\norm{u}_{L^p((\tau,\tau_{u_0});E^\gamma)}=\infty$,
which contradicts \eqref{e:INTEGRABILITY}.
\end{proof}

An important case of our setting is when $\J_0\subset\J$ are intervals and the regularity index $\reg\colon\J\to\R$ associated with the family $\{E^\sigma\}_{\sigma\in\J}$ is a continuous increasing function. Using $\reg^{-1}$ we can reparametrize the family so that its regularity index becomes an identity on an interval. Therefore, its worth 
rephrasing Corollary~\ref{cor:EXIST} with additional properties from Proposition~\ref{prop:PROPERTIES} in this situation.

\begin{cor}\label{cor:INTERVAL}
Assume that $\J_0\subset\J$ are intervals and $\{E^\sigma\}_{\sigma\in\J}$ is a family of Banach spaces with regularity index being an identity. Assume \ref{a:B1}, \ref{a:B2} and let $\calf_{i}\colon E^{\alpha_i}\to E^{\beta_i}$ satisfy \eqref{e:EPSREG}, \eqref{e:EPSREG2}
with $c_0>0$, $\rho_i\geq 1$, $\beta_i\in\J_0$, $\alpha_i\in\J$ for $i\in\cali$ such that \eqref{e:SETUP1} and
\begin{equation}\label{e:WAR2}
0\leq\alpha_j-\beta_i<1,\ i,j\in\cali.
\end{equation}
Let $\gamma\in\J_0$ satisfy \eqref{e:SETUP2} and \eqref{e:TOGAMMA} with
\begin{equation}\label{e:WAR1EQUIV}
\omega_i=1+\beta_i-\alpha_i\rho_i+(\rho_i-1)\gamma>0\ \text{ for }i\in\cali,
\end{equation}
and
\begin{equation}\label{e:WAR4}
\max_{i\in\cali}\beta_i\leq\gamma\leq\min_{i\in\cali}\alpha_i.
\end{equation}
Assume that $\{S(t)\}_{t>0}$ is a~semigroup on $E^{\gamma}$ and on each $E^{\beta_i}$ for $i\in\cali$. Then for any $u_0\in E^\gamma$ there exists a~$\gamma$-solution $u(\cdot,u_0)$ of \eqref{e:VCF} defined on $(0,\tau_{u_0})$ with the maximal time of existence $\tau_{u_0}$, having regularity as in \eqref{e:HOWREGULAR} and satisfying \eqref{e:LIMSUP} provided that $\tau_{u_0}<\infty$. 
Moreover, it extends continuously in $E^\gamma$ at $t=0$ with $u(0)=u_0$ if \eqref{e:STC0} holds.
Furthermore, if $\delta\in\J$ is such that $\gamma\leadsto\delta$, $\beta_i\leadsto\delta$ for all $i\in\cali$ 
and
$$\gamma\leq\delta<\min_{i\in\cali}\beta_i+1,$$
then
\begin{itemize}
\item[(i)] $u\in C((0,\tau_{u_0});E^\delta)$,
\item[(ii)] $t^{\delta-\gamma}\norm{u(t)}_{E^{\delta}}\to 0$  as $t\to 0^{+}$  provided that 
\begin{equation}\label{e:CONVGAMMATHETA1}
t^{\delta-\gamma}\norm{S(t)u_0}_{E^{\delta}}\to 0\ \text{ as }t\to 0^{+},
\end{equation}
\item[(iii)] $u(t)\to u_0$ in $E^{\delta}$ as $t\to 0^{+}$
if $\delta<\gamma+\min\limits_{i\in\cali}\omega_i$, $u_0\in E^{\delta}$ and $S(t)u_0\to u_0$ in $E^{\delta}$ as $t\to 0^{+}$.
\end{itemize}
\end{cor}

Note that by Proposition~\ref{prop:CONV2} condition \eqref{e:CONVGAMMATHETA1} holds if $\delta\in(\gamma,\min\limits_{i\in\cali}\beta_i+1)$ and there exists $\sigma\in(\gamma,\delta]\cap\J_0$ such that $\sigma\leadsto\delta$ and $E^\sigma\cap E^{\gamma}$ is dense in $E^\gamma$.

Conditions \eqref{e:WAR2}--\eqref{e:WAR4} of Corollary~\ref{cor:INTERVAL} are consistent with the restrictions in the subcritical case considered in \cite{Q15,CHQRB17}  where the nonlinearity was just a~\emph{single} map, i.e., $|\cali|=1$, and the family $\{S(t)\}_{t>0}$ was a smoothing semigroup on the whole Banach scale $\{E^\gamma\}_{\gamma\in\J}$. A~range of admissible spaces from which one can take the initial condition $u_0$ in order to solve \eqref{e:VCF} was derived in \cite[(2.20)]{CHQRB17}, whereas the continuity of solutions in $E^\gamma$ was discussed in \cite[Proposition 2.10]{CHQRB17}. 

\section{Regularization of $\gamma$-solutions}\label{sec:REGULARIZATION}

Under the assumptions of Corollary~\ref{cor:EXIST}, 
the $\gamma$-solution constructed there satisfies
\begin{equation}\label{e:NEWFORMULA}
u(t)=S\left(\tfrac{t}{2}\right)u\left(\tfrac{t}{2}\right)+\sum_{i\in\cali}\int_{\frac{t}{2}}^{t}S(t-s)\calf_i(u(s))ds,\ t\in(0,\tau_{u_0}).
\end{equation}
Indeed, since the family is a semigroup on $E^\gamma$ and each $E^{\beta_i}$ and $S(\frac{t}{2})\in\mathcal{L}(E^{\gamma};E^{\gamma}\cap\bigcap\limits_{i\in\cali}E^{\alpha_i})$, we have for $t\in(0,\tau_{u_0})$
\begin{equation*}
u(t)=S\left(\tfrac{t}{2}\right)S\left(\tfrac{t}{2}\right)u_0+\sum_{i\in\cali}S\left(\tfrac{t}{2}\right)\int_{0}^{\frac{t}{2}}S\left(\tfrac{t}{2}-s\right)\calf_i(u(s))ds+\sum_{i\in\cali}\int_{\frac{t}{2}}^{t}S(t-s)\calf_i(u(s))ds.
\end{equation*}
Using the linearity of $S(\frac{t}{2})$ on $E^{\gamma}$, we get \eqref{e:NEWFORMULA}.

Having \eqref{e:NEWFORMULA}, in certain situations we can improve regularity of such solution under relatively mild assumptions in comparison to Lemmas~\ref{lem:REGH} and \ref{lem:REGH2}. 

\begin{lem}\label{lem:BETTERREGULARITY}
Let $\calf_i$, $i\in\cali$, satisfy \eqref{e:EPSREG2}, i.e.,
\begin{equation*}
\norm{\calf_i(\phi)}_{E^{\beta_i}}\leq c_0(\norm{\phi}_{E^{\alpha_i}}^{\rho_i}+1),\ \phi\in E^{\alpha_i},
\end{equation*}
with $c_0>0$, $\rho_i\geq 1$, $\alpha_i\in\J$, $\beta_i\in\J_0$ for $i\in\cali$ and let $\gamma\in\J_0$ be such that
\begin{equation}\label{e:WEAKERSETUP2}
\reg(\gamma)\leq\reg(\alpha_i)\ \text{ for all }i\in\cali.
\end{equation}
Assume that
\begin{equation*}
u\in \bigcap_{T\in\calt_\tau}\bigcap_{i\in\cali}\mathcal{L}^\infty_{\dif(\gamma,\alpha_i)}((0,T];E^{\alpha_i})
\end{equation*}
satisfies 
\begin{equation}\label{e:NEWFORMULAGENERAL}
u(t)=S\left(\tfrac{t}{2}\right)u\left(\tfrac{t}{2}\right)+\sum_{i\in\cali}\int_{\frac{t}{2}}^{t}S(t-s)\calf_i(u(s))ds,\ t\in\calt_\tau.
\end{equation}
Let $\delta\in\J$ be such that \eqref{e:THETA} holds, i.e.,
\begin{equation*}
\beta_i\leadsto\delta\ \text{ and }\ \dif(\beta_i,\delta)<1\text{ for all  }i\in\cali,
\end{equation*}
and 
\begin{equation}\label{e:ALPHAJDELTA}
\J_0\ni\alpha_j\leadsto\delta\text{ for some }j\in\cali.
\end{equation}
(i) If \eqref{e:NONNEGATIVEOMEGA} holds, i.e., 
$$\omega_i=1-\rho_i\dif(\gamma,\alpha_i)+\reg(\beta_i)-\reg(\gamma)\geq 0\text{ for all }i\in\cali,$$
then for any $T\in\calt_\tau$ we have $u\in\mathcal{L}^\infty_{\dif(\gamma,\delta)}((0,T];E^{\delta})$ and
\begin{equation}\label{e:BETTERREGULARITY}
\begin{split}
\norm{u}_{\delta,T}&=\sup_{t\in(0,T]}t^{\dif(\gamma,\delta)}\|u(t)\|_{E^{\delta}}\leq 2^{\dif(\gamma,\delta)} M(\alpha_j,\delta,T)\norm{u}_{\alpha_j,T}\\
&+c_0\sum_{i\in\cali}\tfrac{M(\beta_i,\delta,T)}{2^{1-\dif(\beta_i,\delta)}(1-\dif(\beta_i,\delta))}\left(T^{1-\dif(\beta_i,\delta)+\dif(\gamma,\delta)}+2^{\rho_i\dif(\gamma,\alpha_i)}\norm{u}_{\alpha_i,T}^{\rho_i}T^{\omega_i}\right).
\end{split}
\end{equation}
(ii) If $u\in C(\calt_\tau;E^{\alpha_j})$ then $u\in C(\calt_\tau;E^\delta)$.
\end{lem}

\begin{proof}
$(i)$ Note that $\reg(\delta)-\reg(\gamma)=\dif(\alpha_j,\delta)+\dif(\gamma,\alpha_j)\geq 0$. For any $T\in\calt_\tau$ 
we estimate \eqref{e:NEWFORMULAGENERAL} applying \eqref{e:INTEGRALBASIC0} and using \eqref{e:EPSREG2}, \eqref{e:THETA} and \eqref{e:ALPHAJDELTA} to obtain
\begin{equation*}
\begin{split}
t^{\dif(\gamma,\delta)}\|u(t)\|_{E^{\delta}}&\leq 2^{\dif(\gamma,\delta)}M(\alpha_j,\delta,T)\bigl(\tfrac{t}{2}\bigr)^{\dif(\gamma,\alpha_j)}\|u\bigl(\tfrac{t}{2}\bigr)\|_{E^{\alpha_j}}\\
&+c_0\sum_{i\in\cali}\tfrac{M(\beta_i,\delta,T)}{2^{1-\dif(\beta_i,\delta)}(1-\dif(\beta_i,\delta))}\Bigl(t^{1+\reg(\beta_i)-\reg(\gamma)}+2^{\rho_i\dif(\gamma,\alpha_i)}\norm{u}_{\alpha_i,T}^{\rho_i}t^{\omega_i}\Bigr).
\end{split}
\end{equation*}
Since $\reg(\gamma)-\reg(\beta_i)=\dif(\beta_i,\delta)-\dif(\gamma,\delta)<1$, by \eqref{e:NONNEGATIVEOMEGA} we get \eqref{e:BETTERREGULARITY}.

\noindent 
$(ii)$ Observe that for $0<t_0<T\in\calt_\tau$ and $t\in[\frac{t_0}{2},T]$ from \eqref{e:NEWFORMULAGENERAL} we get
\begin{equation}\label{e:INDELTA1}
\begin{split}
&\|u(t)-u(t_0)\|_{E^{\delta}}\leq\|S(\tfrac{t}{2})[u(\tfrac{t}{2})-u(\tfrac{t_0}{2})]\|_{E^{\delta}}+\|(S(\tfrac{t}{2})-S(\tfrac{t_0}{2}))u(\tfrac{t_0}{2})\|_{E^{\delta}}\\
&+\Bigl\|\sum_{i\in\cali}\int_{\frac{t}{2}}^{t}S(t-s)\calf_i(u(s))ds-\sum_{i\in\cali}\int_{\frac{t_0}{2}}^{t_0}S(t_0-s)\calf_i(u(s))ds\Bigr\|_{E^{\delta}}.
\end{split}
\end{equation}
Since $\alpha_j\leadsto\delta$ and $u\in C(\calt_\tau;E^{\alpha_j})$ for some $j\in\cali$, we have
$$\|(S(\tfrac{t}{2})-S(\tfrac{t_0}{2}))u(\tfrac{t_0}{2})\|_{E^{\delta}}\to 0\text{ as }t\to t_0,$$
$$\|S(\tfrac{t}{2})[u(\tfrac{t}{2})-u(\tfrac{t_0}{2})]\|_{E^{\delta}}\leq M(\alpha_j,\delta,T)(\tfrac{2}{t})^{\dif(\alpha_j,\delta)}\|u(\tfrac{t}{2})-u(\tfrac{t_0}{2})]\|_{E^{\alpha_j}}\to 0\text{ as }t\to t_0.$$
Thus, to prove the claim, it suffices to show that the last term in the right-hand side of \eqref{e:INDELTA1}, denoted $I_t$, tends to zero as $t\to t_0$.

We first show that $I_t$ converges to zero as $t\to t_0^{+}$. For $t_0<t<\min\{2t_0,T\}$, by \eqref{e:THETA} and \eqref{e:INTEGRALBASIC0} with $\nu_i(\delta)=1-\dif(\beta_i,\delta)\in(0,1]$ and $\mu_i=1-\rho_i\dif(\gamma,\alpha_i)$, we get
\begin{equation*}
\begin{split}
&I_t\leq 2c_0\sum_{i\in\cali}M(\beta_i,\delta,T)\tfrac{(t-t_0)^{\nu_i(\delta)}}{\nu_i(\delta)}
\left(1+\|u\|_{\alpha_i,T}^{\rho_i}\bigl(\tfrac{t_0}{2}\bigr)^{\mu_i-1}\right)+\int_{\frac{t_0}{2}}^{t_0}g_t(s)ds
\end{split}
\end{equation*}
with 
$$g_t(s):=\Bigl\|\sum_{i\in\cali}(S(t-s)-S(t_0-s))\calf_i(u(s))\Bigr\|_{E^{\delta}}\ \text{ for }\ s\in(\tfrac{t_0}{2},t_0).$$
Hence it remains to show that
\begin{equation}\label{e:INTGTGOESTOZERO}
\int_{\frac{t_0}{2}}^{t_0}g_t(s)ds\to 0\ \text{ as }\ t\to t_0^{+}.
\end{equation} 
Using \eqref{e:EPSREG2} and \eqref{e:THETA} with $0<\nu_i(\delta)\leq 1$, we obtain for $s\in(\frac{t_0}{2},t_0)$ 
\begin{equation*}
g_t(s)
\leq 2c_0\sum_{i\in\cali}M(\beta_i,\delta,\tau)(t_0-s)^{\nu_i(\delta)-1}\left(1+\norm{u}_{\alpha_i,T}^{\rho_i}\bigl(\tfrac{t_0}{2}\bigr)^{\mu_i-1}\right)=:g(s).
\end{equation*}
Hence $g_t$ for $t\in(t_0,\min\{2t_0,T\})$ is bounded by the integrable function $g$ on $(\tfrac{t_0}{2},t_0)$. 
Moreover, given $s\in(\frac{t_0}{2},t_0)$, we have
\begin{equation*}
0\leq g_t(s)\leq
\sum_{i\in\cali}\norm{(S(t-s)-S(t_0-s))\calf_i(u(s))}_{E^{\delta}}\to 0\ \text{ as }\ t\to t_0^+,
\end{equation*}
since $\calf_i(u(s))\in E^{\beta_i}$ and $\beta_i\leadsto\delta$.
Consequently, Lebesgue's dominated convergence theorem yields \eqref{e:INTGTGOESTOZERO}.
This shows that $\norm{u(t)-u(t_0)}_{E^{\delta}}\to 0$ as $t\to t_0^{+}$.

We will now prove that $I_t\to 0$ as $t\to t_0^{-}$. Given $\eps>0$, we choose $0<\eta\leq\frac{t_0}{2}$ so small that 
$$2c_0|\cali|\max_{i\in\cali}M(\beta_i,\delta,T)\tfrac{\eta^{\nu_i(\delta)}}{\nu_i(\delta)}	\Bigl(1+\|u\|_{\alpha_i,T}^{\rho_i}\bigl(\tfrac{t_0}{2}\bigr)^{\mu_i-1}\Bigr)\leq \eps.$$
Let  
$t\in[\frac{t_0}{2},T]$ be such that $0<t_0-t<\eta$. Then we have $t\in(\frac{t_0}{2},t_0)$ and by  \eqref{e:INTEGRALBASIC0} we get
\begin{equation*}
\begin{split}
I_t&\leq 2c_0\sum_{i\in\cali}M(\beta_i,\delta,T)\tfrac{(t_0-t)^{\nu_i(\delta)}}{\nu_i(\delta)}
\Bigl(1+\|u\|_{\alpha_i,T}^{\rho_i}\bigl(\tfrac{t}{2}\bigr)^{\mu_i-1}\Bigr)\\
&+2c_0\sum_{i\in\cali}M(\beta_i,\delta,T)\tfrac{\eta^{\nu_i(\delta)}}{\nu_i(\delta)}
\Bigl(1+\|u\|_{\alpha_i,T}^{\rho_i}\bigl(\tfrac{t_0}{2}\bigr)^{\mu_i-1}\Bigr)+\int_{\frac{t_0}{2}}^{t_0-\eta}h_{t}(s)ds,\\
\end{split}
\end{equation*}
where 
\begin{equation*}
\begin{split}
h_t(s)&:=\Bigl\|\sum_{i\in\cali}(S(t-s)-S(t_0-s))\calf_i(u(s))\Bigr\|_{E^{\delta}}\\ &\leq 2c_0\sum_{i\in\cali}M(\beta_i,\delta,T)(t_0-\eta-s)^{\nu_i(\delta)-1}\Bigl(1+\|u\|_{\alpha_i,T}^{\rho_i}\bigl(\tfrac{t_0}{2}\bigr)^{\mu_i-1}\Bigr)=:h(s).
\end{split}
\end{equation*}
 Note that $h_t$ is bounded by the integrable function $h$ on $(\frac{t_0}{2},t_0-\eta)$ and for  $s\in(\frac{t_0}{2},t_0-\eta)$
\begin{equation*}
\begin{split}	
&0\leq h_t(s)\leq\sum_{i\in\cali}\norm{(S(t-s)-S(t_0-s))\calf_i(u(s))}_{E^{\delta}}\to 0\ \text{ as }\ t\to t_0^-.
\end{split}
\end{equation*}
Hence Lebesgue's dominated convergence theorem implies
\begin{equation*}
\int_{\frac{t_0}{2}}^{t_0-\eta}h_t(s)ds\to 0\ \text{ as }t\to t_0^{-}.
\end{equation*}
This shows that $\norm{u(t)-u(t_0)}_{E^{\delta}}\to 0$  as $t\to t_0^{-}$,
which ends the proof.
\end{proof}

\begin{rem}\label{rem:ITERATION}
(i) If instead of \eqref{e:EPSREG2}, we have
\begin{equation*}
\norm{\calf_i(\phi)}_{E^{\beta_i}}\leq c_0\norm{\phi}_{E^{\alpha_i}}^{\rho_i},\ \phi\in E^{\alpha_i},\ i\in\cali,
\end{equation*}
then the estimate in \eqref{e:BETTERREGULARITY} simplifies to
\begin{equation*}
\|u\|_{\delta,T}\leq 2^{\dif(\gamma,\delta)} M(\alpha_j,\delta,T)\norm{u}_{\alpha_j,T}+c_0\sum_{i\in\cali}\tfrac{M(\beta_i,\delta,T)}{2^{1-\dif(\beta_i,\delta)-\rho_i\dif(\gamma,\alpha_i)}(1-\dif(\beta_i,\delta))}\norm{u}_{\alpha_i,T}^{\rho_i}T^{\omega_i}.
\end{equation*}
(ii) If we apply the above lemma with some $\delta=\alpha_j'\in\J_0$ for $j\in\cali$, we will be able to reiterate the procedure of Lemma~\ref{lem:BETTERREGULARITY} with $\alpha_i'$ in the role of $\alpha_i$ provided that suitable estimates \eqref{e:EPSREG2} of $\calf_i$ between $E^{\alpha_i'}$ and some $E^{\beta_i'}$ are available and the counterparts of  \eqref{e:WEAKERSETUP2} and \eqref{e:NONNEGATIVEOMEGA} hold.
\end{rem} 

\section{Extrapolated fractional power scales}\label{sec:FRACTIONAL}

As an example of the setting used in previous sections serves the theory of extrapolated fractional power scales, see e.g. \cite{AM,Las,Q15,CzD}. 

Let $A\colon X\supset\dom(A)\to X$ be a positive sectorial operator in a reflexive Banach space $(X,\norm{\cdot}_{X})$. We denote by
$\iota$ the dense and continuous embedding of $D(A^*)$ into $X^*$, where $D(A^*)$ is the domain $\dom(A^*)$ of the adjoint $A^*$ of $A$ endowed with the graph norm.  
If $j\in\mathcal{L}(X;[D(A^*)]^{*})$ is a composition of the canonical isometry between $X$ and $X^{**}$ with the adjoint map of $\iota$, then $j$ becomes an isometric embedding of the space $X$ endowed with the norm $\|A^{-1}\cdot\|_{X}$ into $[D(A^*)]^{*}$ with dense image $j(X)$. 
Hence $E^{-1}=[D(A^*)]^{*}$ with $\|\cdot\|_{E^{-1}}=\|\cdot\|_{[D(A^*)]^{*}}$ is a completion of the space $X$ endowed with the norm $\|A^{-1}\cdot\|_{X}$ via $j$. We call $E^{-1}$ the extrapolated Banach space corresponding to $A$ and set $E^0=j(X)$.

Let $A_{-1}\colon E^{-1}\supset E^0\to E^{-1}$ be the corresponding extrapolated operator, which is a~unique linear extension of $A$ such that $A_{-1}\circ j\in\mathcal{L}iso(X;E^{-1})$ is an isometric linear isomorphism from $X$ onto $E^{-1}$. 
The extrapolated operator $A_{-1}$ is also a positive sectorial operator in $E^{-1}$. In particular, $-A_{-1}$ generates an analytic $C^0$ semigroup $\{S(t)=e^{-A_{-1}t}\}_{t\geq 0}$ in $E^{-1}$ with $S(0)$ being the identity in $E^{-1}$. Using fractional power spaces for $A_{-1}$, we obtain a scale of reflexive Banach spaces $E^\sigma=D((A_{-1})^{\sigma+1})$ for $\sigma\geq -1$ endowed with the norms
$$\|x\|_{E^\sigma}=\|(A_{-1})^{\sigma+1}x\|_{E^{-1}},\ x\in E^\sigma,\ \sigma\geq -1.$$
The scale is dense and nested, that is, $E^\xi$ is densely and continuously embedded into $E^\sigma$ for $\xi\geq\sigma\geq -1$.
Moreover, we get operators $A_{\sigma}\colon E^\sigma\supset E^{\sigma+1}\to E^\sigma$ such that $A_\sigma x=A_{-1}x$, $x\in E^{\sigma+1}$
and $A_\sigma\in\mathcal{L}iso(E^{\sigma+1};E^{\sigma})$. In fact, we have
\begin{equation}\label{e:ISOMETRY}
(A_\sigma)^\theta\in\mathcal{L}iso(E^{\sigma+\theta};E^\sigma),\ \sigma\geq -1,\ \theta\geq 0.
\end{equation}
The operators $A_\sigma$, $\sigma\geq -1$, are positive sectorial operators in $E^\sigma$ and each $-A_\sigma$ generates an analytic $C^0$ semigroup $\{e^{-A_\sigma t}\}_{t\geq 0}$ in $E^\sigma$ such that $e^{-A_\sigma t}$ is a restriction of $S(t)$ to $E^\sigma$ and for some $a>0$
\begin{equation}\label{e:SMOOTHEXT}
t^{\xi-\sigma}\|S(t)x\|_{E^\xi}\leq C_{\sigma,\xi}e^{-a t}\norm{x}_{E^\sigma},\ x\in E^\sigma,\ \xi\geq\sigma\geq -1,\ t>0,
\end{equation}
with some constants $C_{\sigma,\xi}>0$. Thus we see that assumptions \ref{a:B1}, \ref{a:B2} 
are satisfied with $\{S(t)\}_{t\geq 0}$ and $\{E^\sigma\}_{\sigma\in\J}$, $\J=[-1,\infty)$, with regularity index being an identity. We choose an interval $\J_0\subset\J=[-1,\infty)$ and note that for any $\xi\geq\sigma\in\J_0$ we have $\sigma\leadsto\xi$ in the sense of Definition~\ref{def:SMOOTHS}.

\begin{ass}\label{e:ASS}
Let $\calf_{i}\colon E^{\alpha_i}\to E^{\beta_i}$, $i\in\cali$, satisfy \eqref{e:EPSREG}, \eqref{e:EPSREG2} with $c_0>0$, $\rho_i\geq 1$ and $\alpha_i\geq\beta_i\in\J_0$ for $i\in\cali$ such that \eqref{e:WAR2} holds. Let $\gamma\in\J_0$ satisfy \eqref{e:WAR1EQUIV} and \eqref{e:WAR4} and set $\beta:=\min\limits_{i\in\cali}\beta_i$.
\end{ass}

Under Assumption~\ref{e:ASS}, in the current setting Corollary~\ref{cor:INTERVAL} guarantees
that for any $u_0\in E^\gamma$ there exists a unique $\gamma$-solution $u$ of \eqref{e:VCF} defined on the maximal interval of existence $[0,\tau_{u_0})$, that is, $u(0)=u_0$,  
$$u\in C([0,\tau_{u_0});E^\gamma)\cap C((0,\tau_{u_0});E^{(\beta+1)_{-}})\cap\bigcap_{T\in(0,\tau_{u_0})}\bigcap_{i\in\cali} \mathcal{L}^\infty_{\alpha_i-\gamma}((0,T];E^{\alpha_i}),$$
$u$ satisfies 
the Duhamel formula
\begin{equation}\label{e:EXTVCF1}
u(t)=S(t)u_0+\sum_{i\in\cali}\int_{0}^{t}S(t-s)\calf_i(u(s))ds,\ t\in[0,\tau_{u_0}),
\end{equation} 
and if $\tau_{u_0}<\infty$ then
\begin{equation*}
\lim_{t\to\tau_{u_0}^{-}}\norm{u(t)}_{E^{\gamma}}=\infty.
\end{equation*}	
Moreover, by the density of the scale, \eqref{e:SMOOTHEXT} and Corollary~\ref{cor:INTERVAL}~(ii), if $\theta\in(0,\beta-\gamma+1)$ 
then
\begin{equation*}
t^{\theta}\norm{u(t)}_{E^{\gamma+\theta}}\to 0\ \text{ as }\ t\to 0^{+}.
\end{equation*}

We now show that the $\gamma$-solution exhibits further time regularity properties in this special situation. 

\begin{thm}\label{thm:EXTRAPOLATED}
In the setting of extrapolated fractional power scale, under Assumption~\ref{e:ASS} the $\gamma$-solution $u$ of \eqref{e:EXTVCF1} fulfills  
$$u\in C((0,\tau_{u_0});E^{\beta+1}),\ \dot{u}\in C((0,\tau_{u_0});E^\xi)\ \text{ for }\ \xi<\beta+1,$$ 
and satisfies in $E^\beta$ the differential equation
\begin{equation}\label{e:DIFFEQ}
\dot{u}(t)+A_\beta u(t)=\sum_{i\in\cali}\calf_i(u(t)),\ t\in(0,\tau_{u_0}).
\end{equation}
\end{thm}
   
\begin{proof}
We adapt the argument presented in the proof of \cite[Lemma 2.2.1]{C-D}.

\noindent 
\textit{Step 1.} Given $\sigma\in\J=[-1,\infty)$ and $\theta\in(0,1]$ we have for $v\in D(A_\sigma^\theta)=E^{\sigma+\theta}$ and $h>0$
$$(e^{-A_\sigma h}-I)v=\int_{0}^{h}\tfrac{d}{ds}(e^{-A_\sigma s}v)ds=-\int_{0}^{h}A_\sigma e^{-A_\sigma s}vds=-\int_{0}^{h}A_\sigma^{1-\theta}e^{-A_\sigma s}A_\sigma^\theta vds.$$
Thus, by \eqref{e:SMOOTHEXT} we get  for $v\in E^{\sigma+\theta}$ and $h>0$
\begin{equation}\label{e:DIFFTOZERO}
\begin{split}
&\|(e^{-A_\sigma h}-I)v\|_{E^\sigma}\leq
\int_{0}^{h}\norm{e^{-A_\sigma s}A_\sigma^\theta v}_{E^{\sigma+1-\theta}}ds
\leq C_{\sigma,\sigma+1-\theta}\theta^{-1}\norm{v}_{E^{\sigma+\theta}}h^\theta.
\end{split}
\end{equation}

\noindent 
\textit{Step 2.} Observe that for $0<\eps<\tau_{u_0}$ and $0\leq t<\tau_{u_0}-\eps$ we have
\begin{equation}\label{e:SHIFTEPS}
u(t+\eps)=S(t)u(\eps)+\sum_{i\in\cali}\int_0^t S(t-s)\calf_i(u(s+\eps))ds.
\end{equation} 
Let $0<T_1<T_2<\tau_{u_0}$ and choose $0<\eps<T_1$. For $0<h<T_2-\eps$ and $t\in(0,T_2-h-\eps]$, we get 
$$u(t+h+\eps)-u(t+\eps)=(S(h)-I)S(t)u(\eps)+\sum_{i\in\cali}\int_0^h S(t+h-s)\calf_i(u(s+\eps))ds$$
$$+\sum_{i\in\cali}\int_0^{t} S(t-s)[\calf_i(u(s+h+\eps))-\calf_i(u(s+\eps))]ds=\calj_1+\calj_2+\calj_3.$$
Let $\theta\in(0,1)$. Set $\alpha:=\max\limits_{i\in\cali}\alpha_i$ and note that $0\leq\alpha-\beta<1$ by \eqref{e:WAR2}. 
Using \eqref{e:EPSREG}, \eqref{e:EPSREG2}, \eqref{e:SMOOTHEXT} and \eqref{e:DIFFTOZERO}, we estimate each term in $E^{\alpha}$,
$$\norm{\calj_1}_{E^{\alpha}}=\norm{(S(h)-I)S(t)u(\eps)}_{E^{\alpha}}\leq C_{\alpha,\alpha+1-\theta}C_{\alpha,\alpha+\theta}\theta^{-1}\norm{u(\eps)}_{E^{\alpha}}\tfrac{h^\theta}{t^{\theta}}=\tfrac{c_1h^\theta}{t^\theta},$$
$$\norm{\calj_2}_{E^{\alpha}}
\leq c_0\sum_{i\in\cali}C_{\beta_i,\alpha}\big(\norm{u}_{C([\eps,T_2];E^{\alpha_i})}^{\rho_i}+1
\big)\int_0^h\tfrac{1}{(t+h-s)^{\alpha-\beta_i}}ds\leq\tfrac{c_2 h}{t^{\alpha-\beta}},$$
$$\norm{\calj_3}_{E^{\alpha}}\leq\int_0^{t} \tfrac{c_3}{(t-s)^{\alpha-\beta}}\norm{u(s+h+\eps)-u(s+\eps)}_{E^{\alpha}}ds.$$
Combining these estimates, we obtain with $\eta=\max\{\theta,\alpha-\beta\}\in(0,1)$
$$\norm{u(t+h+\eps)-u(t+\eps)}_{E^{\alpha}}\leq \tfrac{c_4 h^\theta}{t^\eta}+\int_0^{t}\tfrac{c_3}{(t-s)^{\alpha-\beta}}\norm{u(s+h+\eps)-u(s+\eps)}_{E^{\alpha}}ds.$$
By the Volterra type inequality (cf. \cite[Lemma 1.2.9]{C-D}) we see that with some $\bar{c}>0$
$$\sup_{t\in[\eps,T_2-h]}(t-\eps)^\eta\norm{u(t+h)-u(t)}_{E^{\alpha}}\leq\bar{c}h^\theta.$$
Consequently, for any $0<T_1<T_2<\tau_{u_0}$ and $\theta\in(0,1)$ there exists $c_{T_1,T_2,\theta}>0$ such that
\begin{equation}\label{e:HOLDER}
\max_{i\in\cali}\norm{u(t_2)-u(t_1)}_{E^{\alpha_i}}\leq c_{T_1,T_2,\theta}|t_2-t_1|^\theta,\ t_1,t_2\in[T_1,T_2].
\end{equation}

\noindent 
\textit{Step 3.} We will now show that $u(t)\in D(A_\beta)=E^{\beta+1}$ for $t\in(0,\tau_{u_0})$. It suffices to prove that for any $0<\eps<\tau_{u_0}$ we have $u(t+\eps)\in D(A_\beta)$ for $t\in(0,\tau_{u_0}-\eps)$. Taking \eqref{e:SHIFTEPS} into account,  
we will first show that for any $i\in\cali$
\begin{equation}\label{e:INDOMAIN}
\int_0^t S(t-s)\calf_i(u(s+\eps))ds\in D(A_{\beta_i})=E^{\beta_i+1}
\end{equation}
and
\begin{equation}\label{e:UNDERINTEGRAL}
A_{\beta_i}\int_0^t S(t-s)\calf_i(u(s+\eps))ds=\int_0^t A_{\beta_i}S(t-s)\calf_i(u(s+\eps))ds,\ t\in(0,\tau_{u_0}-\eps).
\end{equation}
Indeed, since the semigroup is $C^0$ in $E^{\beta_i}$, the function $[0,t]\ni s\mapsto S(t-s)\calf_i(u(s+\eps))ds\in E^{\beta_i}$ is continuous, hence integrable on $[0,t]$. Moreover, we have
$$S(t-s)\calf_{i}(u(r))
\in D(A_{\beta_i})=E^{\beta_i+1}\ \text{ for }\ s\in[0,t),\ r\in(0,\tau_{u_0})$$
and the function $[0,t)\ni s\mapsto A_{\beta_i}S(t-s)\calf_i(u(s+\eps))ds\in E^{\beta_i}$ is continuous. Then, given $\theta\in(0,1)$,
for any $0<t\leq T<\tau_{u_0}-\eps$ by \eqref{e:EPSREG}, \eqref{e:SMOOTHEXT} and \eqref{e:HOLDER} we have 
$$\int_{0}^{t}\|A_{\beta_i}S(t-s)(\calf_i(u(s+\eps))-\calf_i(u(t+\eps)))\|_{E^{\beta_i}}ds$$
$$\leq c_0C_{\beta_i,\beta_i+1}\int_{0}^{t}\tfrac{1}{t-s}\norm{u(t+\eps)-u(s+\eps)}_{E^{\alpha_i}}(1+\norm{u(t+\eps)}_{E^{\alpha_i}}^{\rho_i-1}+\norm{u(s+\eps)}_{E^{\alpha_i}}^{\rho_i-1})ds$$
$$\leq c_0  C_{\beta_i,\beta_i+1}c_{\eps,T+\eps,\theta}\big(1+2\norm{u}_{C([\eps,T+\eps];E^{\alpha_i})}^{\rho_i-1}\big)\theta^{-1}t^{\theta}.$$
Since
$$\int_{0}^{t}A_{\beta_i}S(t-s)\calf_i(u(t+\eps))ds=\int_{0}^{t}\tfrac{d}{ds}\left(S(t-s)\calf_i(u(t+\eps))\right)ds=(I-S(t))\calf_i(u(t+\eps)),$$
the improper integral
\begin{equation}\label{e:IMPROPER}
\begin{split}
\int_{0}^{t}A_{\beta_i}S(t-s)\calf_i(u(s+\eps))ds&=\int_{0}^{t}A_{\beta_i}S(t-s)(\calf_i(u(s+\eps))-\calf_i(u(t+\eps)))ds\\
&+(I-S(t))\calf_i(u(t+\eps))
\end{split}
\end{equation}
exists in $E^{\beta_i}$. This implies \eqref{e:INDOMAIN} and \eqref{e:UNDERINTEGRAL}.

Since the scale is nested and $S(t)u(\eps)\in D(A_\beta)$, we get $u(t+\eps)\in D(A_\beta)=E^{\beta+1}$ and
$$A_\beta u(t+\eps)=A_\beta S(t)u(\eps)+\sum_{i\in\cali}\int_0^t A_\beta S(t-s)\calf_i(u(s+\eps))ds,\ t\in (0,\tau_{u_0}-\eps).$$ 
We have also shown that
\begin{equation*}
h_i(t)=\int_0^t S(t-s)(\calf_i(u(s+\eps))-\calf_i(u(t+\eps)))ds,\ t\in(0,\tau_{u_0}-\eps),
\end{equation*}
has values in $D(A_{\beta_i})=E^{\beta_i+1}$ and 
\begin{equation}\label{e:AHI}
A_{\beta_i}h_i(t)=\int_0^t A_{\beta_i}S(t-s)(\calf_i(u(s+\eps))-\calf_i(u(t+\eps)))ds.
\end{equation}

\noindent 
\textit{Step 4.} We will prove that for any $0<T<\tau_{u_0}-\eps$,  $\xi<\beta_i+1$ and $\eta\in(0,\beta_i+1-\xi)$ there exists $\bar{c}_{T,\xi,\eta}>0$ such that
\begin{equation}\label{e:HOLDER2}
\norm{A_{\beta_i}h_i(t_2)-A_{\beta_i}h_i(t_1)}_{E^{\xi}}\leq \bar{c}_{T,\xi,\eta}\abs{t_2-t_1}^{\eta},\ t_1,t_2\in(0,T].
\end{equation}
Since the scale is nested, it suffices to show \eqref{e:HOLDER2} for $\beta_i\leq\xi<\beta_i+1$.
We have for $0<t_1<t_2\leq T$
\begin{equation*}
\begin{split}
A_{\beta_i}h_i(t_2)-A_{\beta_i}h_i(t_1)&=\int_0^{t_1}A_{\beta_i}[S(t_2-s)-S(t_1-s)](\calf_i(u(s+\eps))-\calf_i(u(t_1+\eps)))ds\\
&+S(t_2-t_1)[I-S(t_1)](\calf_i(u(t_1+\eps))-\calf_i(u(t_2+\eps)))\\
&+\int_{t_1}^{t_2}A_{\beta_i}S(t_2-s)(\calf_i(u(s+\eps))-\calf_i(u(t_2+\eps))ds=\calj_1+\calj_2+\calj_3.
\end{split}
\end{equation*}
We estimate each of these terms using \eqref{e:ISOMETRY}, \eqref{e:DIFFTOZERO} and \eqref{e:HOLDER}. Taking $\theta\in(\eta+\xi-\beta_i,1)$, we get
\begin{equation*}
\begin{split}
\norm{\calj_1}_{E^\xi}&\leq\int_{0}^{t_1}\|[S(t_2-t_1)-I](A_{-1})^{\xi+2}S(t_1-s)(\calf_i(u(s+\eps))-\calf_i(u(t_1+\eps)))\|_{E^{-1}}ds\\
&\leq c_0C_{-1,-\eta}C_{\beta_i,\eta+\xi+1}\eta^{-1}(t_2-t_1)^{\eta}\int_0^{t_1}\tfrac{1}{(t_1-s)^{\eta+\xi+1-\beta_i}}\norm{u(s+\eps)-u(t_1+\eps)}_{E^{\alpha_i}}\\
&\qquad\qquad\qquad\qquad\qquad\qquad\qquad\qquad\cdot(1+\norm{u(s+\eps)}_{E^{\alpha_i}}^{\rho_i-1}+\norm{u(t_1+\eps)}_{E^{\alpha_i}}^{\rho_i-1})ds\\
&\leq c_0 C_{-1,-\eta}C_{\beta_i,\eta+\xi+1}\eta^{-1}c_{\eps,T+\eps,\theta}\big(1+2\norm{u}_{C([\eps,T+\eps];E^{\alpha_i})}^{\rho_i-1}\big)\tfrac{T^{\theta-\eta-\xi+\beta_i}}{\theta-\eta-\xi+\beta_i}(t_2-t_1)^{\eta}.
\end{split} 
\end{equation*}
Moreover, $\calj_2$ and $\calj_3$ are estimated as follows:
\begin{equation*}
\begin{split}
\|\calj_2\|_{E^\xi}&\leq C_{\beta_i,\xi}(t_2-t_1)^{\beta_i-\xi}\|[S(t_1)-I](\calf_i(u(t_1+\eps))-\calf_i(u(t_2+\eps)))\|_{E^{\beta_i}}\\
&\leq c_0C_{\beta_i,\xi} (C_{\beta_i,\beta_i}+1)\big(1+2\norm{u}_{C([\eps,T+\eps];E^{\alpha_i})}^{\rho_i-1}\big)c_{\eps,T+\eps,\eta+\xi-\beta_i}(t_2-t_1)^\eta,
\end{split}
\end{equation*}
\begin{equation*}
\begin{split}
\|\calj_3\|_{E^\xi}&\leq\int_{t_1}^{t_2}\|S(t_2-s)(\calf_i(u(s+\eps))-\calf_i(u(t_2+\eps)))\|_{E^{\xi+1}}ds\\
&\leq c_0C_{\beta_i,\xi+1}\big(1+2\norm{u}_{C([\eps,T+\eps];E^{\alpha_i})}^{\rho_i-1}\big)c_{\eps,T+\eps,\eta+\xi-\beta_i}\eta^{-1}(t_2-t_1)^\eta.
\end{split}
\end{equation*}
Combining these bounds, we obtain \eqref{e:HOLDER2}.

\noindent 
\textit{Step 5.} Let $0<t<T<\tau_{u_0}$. For $0<h<T-t$ we have
\begin{equation}\label{e:ILORAZ}
\tfrac{u(t+h)-u(t)}{h}-\tfrac{1}{h}(S(h)-I)u(t)=\sum_{i\in\cali}\tfrac{1}{h}\int_{t}^{t+h}S(t+h-s)\calf_i(u(s))ds.
\end{equation}
Since $u(t)\in D(A_\beta)=E^{\beta+1}$, we see that $\frac{1}{h}(S(h)-I)u(t)\to -A_\beta u(t)$ in $E^\beta$ as $h\to 0^{+}$. Moreover, 
each $E^{\beta_i}$ is embedded into $E^\beta$ and we have
$$\Big\|\tfrac{1}{h}\int_{t}^{t+h}S(t+h-s)\calf_i(u(s))ds-\calf_i(u(t))\Big\|_{E^{\beta_i}}
\leq \sup_{s\in[t,t+h]}\|S(t+h-s)\calf_i(u(s))-\calf_i(u(t))\|_{E^{\beta_i}}$$
$$\leq \sup_{r\in[0,h]}\|S(r)\calf_i(u(t+h-r))-\calf_i(u(t+h-r))\|_{E^{\beta_i}}+\sup_{s\in[t,t+h]}\|\calf_i(u(s))-\calf_i(u(t))\|_{E^{\beta_i}}.$$
We conclude that the right-hand side tends to zero as $h\to 0^{+}$, since $s\mapsto \calf_i(u(s))\in E^{\beta_i}$ is continuous at $t$ and for the compact subset $K=\calf_i(u([t,T]))$ of $E^{\beta_i}$ the function $[0,T-t]\times K\ni (r,u)\to S(r)u\in E^{\beta_i}$
is uniformly continuous. In consequence, passing to the limit in \eqref{e:ILORAZ} with $h\to 0^{+}$, we obtain in $E^\beta$ 
$$\dot{u}_{+}(t)+A_\beta u(t)=\sum_{i\in\cali}\calf_i(u(t)),\ t\in(0,\tau_{u_0}).$$
Let $0<2\eps<\tau_{u_0}$. Since $u(\eps)\in D(A_\beta)$, by \eqref{e:SHIFTEPS}, \eqref{e:IMPROPER} and \eqref{e:AHI} we have
\begin{equation}\label{e:PRAWOSTRONNAPOCHODNA}
\dot{u}_{+}(t+\eps)
=-S(t)A_\beta u(\eps)-\sum_{i\in\cali}A_{\beta_i}h_i(t)+\sum_{i\in\cali}S(t)\calf_i(u(t+\eps)),\ t\in(0,\tau_{u_0}-\eps).
\end{equation}
Recall that $S(\cdot)A_\beta u(\eps)\in C([0,\infty);E^\beta)$ and $A_{\beta_i}h_i(\cdot), S(\cdot)\calf_i(u(\cdot+\eps))\in C((0,\tau_{u_0}-\eps);E^{\beta_i})$. Hence we infer that $\dot{u}_{+}(\cdot+\eps)\in C((0,\tau_{u_0}-\eps);E^\beta)$.
We define $E^\beta$-valued function
$$w(t)=\int_\eps^t \dot{u}_{+}(s+\eps)ds+u(2\eps),\ t\in[\eps,\tau_{u_0}-\eps),$$
and observe that $w\in C^1([\eps,\tau_{u_0}-\eps);E^\beta)$ with $\dot{w}(t)=\dot{u}_{+}(t+\eps)$ for $t\in[\eps,\tau_{u_0}-\eps)$ and $w(\eps)=u(2\eps)$. Thus for any linear continuous functional $x^{*}\in (E^\beta)^{*}$ the function $\phi\colon[\eps,\tau_{u_0}-\eps)\to\R$ given by
$$\phi(t)=x^{*}(w(t)-u(t+\eps)),\ t\in [\eps,\tau_{u_0}-\eps),$$
is continuous, right-hand side differentiable with $\dot{\phi}_{+}\equiv 0$ and $\phi(\eps)=0$. Therefore, we obtain $\phi\equiv 0$ and, in consequence, $w(t)=u(t+\eps)$ for $t\in[\eps,\tau_{u_0}-\eps)$. This shows that $u\in C^1([2\eps,\tau_{u_0});E^\beta)$.
As $\eps$ can be arbitrarily small, we finally see that $u\in C^1((0,\tau_{u_0});E^\beta)$ and \eqref{e:DIFFEQ} holds.
Moreover, by \eqref{e:PRAWOSTRONNAPOCHODNA} we have
$$\dot{u}(t+\eps)=-S(t)A_\beta u(\eps)-\sum_{i\in\cali}A_{\beta_i}h_i(t)+\sum_{i\in\cali}S(t)\calf_i(u(t+\eps)),\ t\in(0,\tau_{u_0}-\eps).$$
By \eqref{e:HOLDER2} the right-hand side has values and is continuous in $E^\xi$ for $\xi<\beta+1$, so it follows that 
$\dot{u}\in C((0,\tau_{u_0});E^\xi)$, which ends the proof.
\end{proof}

\begin{rem}
(i) The setting of this section allows to consider $\calf=\sum\limits_{i\in\cali}\calf_i$ as a single map acting from $E^\alpha$ to $E^\beta$ with $\alpha=\max\limits_{i\in\cali}\alpha_i$ and $\beta=\min\limits_{i\in\cali}\beta_i$. 
However, our condition \eqref{e:WAR1EQUIV} is less restrictive than its counterpart for the single map $\calf$ satisfying
\eqref{e:EPSREG}, \eqref{e:EPSREG2} with common $\alpha, \beta$ and $\rho=\max\limits_{i\in\cali}\rho_i$.       

\noindent 
(ii) The results of Theorem~\ref{thm:EXTRAPOLATED} agree with the subcritical case of \cite[Theorem 2.2]{ACRB1999} for $\calf_i$ being $\eps_i$-regular maps with $\eps_i>0$, $i\in\cali$. In this case, $\calf_i\colon X^{1+\eps_i}\to X^{\gamma_i}$ is an $\eps_i$-regular map relative to the pair $(X^1,X^0)$ if there exist $\rho_i>1$, $\gamma_i\in(\rho_i\eps_i,1)$, $c_0>0$ and
$$\|\calf_i(\phi)-\calf_i(\psi)\|_{X^{\gamma_i}}\leq c_0\|\phi-\psi\|_{X^{1+\eps_i}}(1+\|\phi\|_{X^{1+\eps_i}}^{\rho_i-1}+\|\psi\|_{X^{1+\eps_i}}^{\rho_i-1}),\ \phi,\psi\in X^{1+\eps_i},$$
where $X^\sigma=E^{\sigma-1}$, $\sigma\geq 0$, are the fractional power spaces. 
Comparing $\gamma$-solutions of Theorem~\ref{thm:EXTRAPOLATED} with $\eps$-regular solutions of \cite{ACRB1999} starting from $u_0\in X^1$, it suffices to take $\gamma=0$, $\beta_i=\gamma_i-1$, $\alpha_i=\eps_i$. Then \eqref{e:WAR1EQUIV} and \eqref{e:WAR4} hold, whereas  
\eqref{e:WAR2} is a consequence of the assumption \cite[(2.5)]{ACRB1999}, which means that
$\alpha-\beta<1$ with $\alpha=\max\limits_{i\in\cali}\alpha_i$ and $\beta=\min\limits_{i\in\cali}\beta_i$. Thus the assumptions used in the subcritical case of \cite{ACRB1999} are a special case of our setting.
\end{rem}

\section{Applications}\label{sec:APPLICATIONS}

\subsection{Modified viscous Cahn-Hilliard equation}

To illustrate the results of Section~\ref{sec:FRACTIONAL}, we consider the Cauchy problem for the modified viscous Cahn-Hilliard equation
\begin{equation}\label{e:mvCH}
\dot{u}=(\delta-\Delta)(\Delta u+f(x,u) -\mu \dot{u}),\ t>0,\ x\in\R^{N},
\end{equation}
subject to the initial condition
\begin{equation}\label{e:INIT}
u(0)=u_0,
\end{equation}
where $\delta,\mu>0$ are positive constants. Note that \eqref{e:mvCH} is a modification of the viscous Cahn-Hilliard equation
\begin{equation*}
\dot{u}+\Delta^2u+\Delta f(x,u) -\mu \Delta \dot{u}=0,\ t>0,\ x\in\R^{N},
\end{equation*}
if we take formally $\delta=0$ in \eqref{e:mvCH}. We also remark that the long-time dynamics of \eqref{e:mvCH} with $\mu=0$
was studied in \cite{CHRB2012c}.    

We assume that $f$ takes the form
\begin{equation}\label{e:F}
f(x,s)=g(x)+m(x)s+\sum_{i=1}^{n}f_i(x,s),\ x\in\R^{N},\ s\in\R,
\end{equation}
with mildly integrable potential $m\colon \R^N\to\R$ belonging to the locally uniform Lebesgue space $L^{r}_{U}(\R^{N})$ with finite norm
\begin{equation}\label{e:MLRU}
\norm{m}_{L^{r}_{U}(\R^{N})}=\sup_{y\in \R^N}\norm{m}_{L^r(B(y,1))}<\infty\ \text{ for some }\, r>\tfrac{N}{2},\ r\geq 2,
\end{equation}
\begin{equation}\label{e:CONDG}
g\in L^{2}(\R^{N}),
\end{equation}
and  $f_i\colon\R^N\times \R\to \R$, $i=1,\ldots,n$, such that
\begin{equation}\label{e:CONDF0}
f_i(x,0) = 0,\ x\in \R^N,
\end{equation}
\begin{equation}\label{e:CONDLIPF0}
|f_{i}(x,s_1)-f_{i}(x,s_2)|\leq c_0 |s_1-s_2|(1+|s_1|^{\rho_i-1} +|s_2|^{\rho_i-1}), \ x\in\R^{N},\ s_1, s_2\in \R,
\end{equation}
where $c_0$ is a certain positive constant and each exponent 
\begin{equation}\label{e:CONDRHO}
\rho_i\geq 1\text{ is arbitrarily large for }N=1,2\text{ and }1\leq \rho_i<\tfrac{N+2}{N-2}\text{ for }N\geq 3.
\end{equation}

Rescalling the time in \eqref{e:mvCH} 
by factor $\mu$ we formally obtain
\begin{equation*}
[1+\kappa(\delta-\Delta)^{-1}]\dot{u}=\Delta u+f(x,u),
\end{equation*}
with $\kappa=\frac{1}{\mu}>0$. 
Since 
$$[1+\kappa(\delta-\Delta)^{-1}]^{-1}=(\delta-\Delta)(\delta+\kappa-\Delta)^{-1}=I-\kappa(\delta+\kappa-\Delta)^{-1}=:B_\kappa,$$
the modified viscous Cahn-Hilliard equation \eqref{e:mvCH} leads to the evolution equation
\begin{equation}\label{e:PROB}
\dot{u}=B_\kappa\Delta u +B_\kappa f(x,u),\ t>0,\ x\in\R^{N}.
\end{equation}
Note that $B_\kappa\Delta u=\Delta u+\kappa u-\kappa(\delta+\kappa)(\delta+\kappa-\Delta)^{-1}u$,
so \eqref{e:PROB} becomes the reaction-diffusion equation
\begin{equation}\label{e:MAIN}
\dot{u}=\Delta u+\kappa u-\kappa(\delta+\kappa)(\delta+\kappa-\Delta)^{-1}u+f(x,u)-\kappa(\delta+\kappa-\Delta)^{-1}f(x,u).
\end{equation}

\begin{rem}
Instead of \eqref{e:mvCH}, we could alternatively consider
\begin{equation}\label{e:m2vCH}
(1-\mu)\dot{u}=(\delta-\Delta)(\Delta u+f(x,u) -\mu \dot{u}),\ t>0,\ x\in\R^{N},
\end{equation}
with $0<\mu\leq 1$ and $\delta>0$. This is a modification of the viscous Cahn-Hilliard equation (for $\delta=0$) in the form
\begin{equation*}
(1-\mu)\dot{u}+\Delta^2u+\Delta f(x,u) -\mu \Delta \dot{u}=0,\ t>0,\ x\in\R^{N}.
\end{equation*}
For $\kappa=\frac{1}{\mu}-1\geq 0$, \eqref{e:m2vCH} again leads to the reaction-diffusion equation \eqref{e:MAIN}.
Note that the asymptotics of \eqref{e:m2vCH} was investigated in \cite{DS10} with more regular $f$ than assumed here. 
\end{rem}

Setting $\cali=\{0,\ldots,2n+1\}$, we rewrite \eqref{e:MAIN} with $\delta>0$, $\kappa\geq 0$, subject to \eqref{e:INIT}, as
\begin{equation}\label{e:ACPVCH}
\begin{cases}
\dot{u}+Au=\sum\limits_{i\in\cali}\calf_i(u),\ t>0,\\
u(0)=u_0,
\end{cases}
\end{equation}
where $Au=(1-\Delta)u$, $\calf_0(u)= m(\cdot)u+(\kappa+1)u+g$, 
$$\calf_i(u)=f_{i}(\cdot,u)\text{ and }\calf_{n+1+i}(u)=-\kappa(\delta+\kappa-\Delta)^{-1}(f_i(\cdot,u))\text{ for }i=1,\ldots,n,$$
$$\calf_{n+1}(u)=-\kappa(\delta+\kappa-\Delta)^{-1}(m(\cdot)u+(\delta+\kappa)u+g).$$
Note that $A$ is a positive sectorial operator in $L^2(\R^{N})$ with $\dom(A)=H^2(\R^{N})$. Thus, as described in Section~\ref{sec:FRACTIONAL}, it generates an extrapolated fractional power scale of Banach spaces $E^\sigma$, $\sigma\in\J=[-1,\infty)$, with corresponding operators $A_\sigma$. In particular, $A$ extends uniquely to a sectorial operator $A_{-1}\colon E^{-1}\supset E^0=L^2(\R^{N})\to E^{-1}$ and $-A_{-1}$ generates an~analytic $C^0$ semigroup $\{S(t)\}_{t\geq 0}$ in $E^{-1}$ satisfying \eqref{e:SMOOTHEXT}. 
Moreover, we have the following characterization of $E^\sigma$ for $\sigma\in\J_0=[-1,1]$:
$$E^{\sigma}=\begin{cases}
H^{2\sigma}(\R^{N}),&\ \sigma\in[0,1],\\
H^{2\sigma}(\R^{N})=(H^{-2\sigma}(\R^{N}))^*,&\ \sigma\in[-1,0).
\end{cases}$$

We decompose $\calf_0=\calf_{01}+\calf_{02}$ into $\calf_{01}u=m(\cdot)u$ and $\calf_{02}u=(\kappa+1)u+g$.
In order to examine its behavior on the scale, we recall from \cite[Appendix A]{CHRB2012c} a property of the multiplication operator by a potential $m\in L^r_{U}(\R^N)$. Below $\|\cdot\|$ denotes the standard norm in $L^2(\R^N)$.

\begin{lem}\label{lem:PROPERTY}
Let $m\in L^{r}_{U}(\R^{N})$ with $r>\frac{N}{2}$, $r\geq 2$.
Then, given 
$\sigma\in[\frac{N}{2r},1)$ such that $\sigma>\frac{N}{2r}$ if $r=2$,
for any $\phi\in H^{2\sigma}(\R^{N})$ the multiplication $m\phi$ uniquely defines an element of $L^2(\R^{N})$ such that
\begin{equation}\label{e:A0}
\|m\phi\|\leq c\|m\|_{L^{r}_{U}(\R^{N})}\|\phi\|_{H^{2\sigma}(\R^{N})},\ \phi\in H^{2\sigma}(\R^{N}).
\end{equation}
\end{lem}

\begin{proof}
Given $\phi\in H^{2\sigma}(\R^{N})$, the functional $L^2(\R^{N})\ni \psi\mapsto \int_{\R^{N}}m\phi\psi dx$
is linear and continuous, since for $p=\frac{2r}{r-2}$, $r\in[2,\infty]$, we have
$$\abs{\int_{\R^{N}}m\phi\psi dx}\leq \sum_{k\in\Z^{N}}\int_{Q_k}|m||\phi||\psi|dx\leq \|m\|_{L^r_U(\R^{N})}\sum_{k\in\Z^{N}}\|\psi\|_{L^2(Q_k)}\|\phi\|_{L^{p}(Q_k)}$$
$$\leq
\|m\|_{L^r_U(\R^{N})}\|\psi\|\Big(\sum_{k\in\Z^{N}}\|\phi\|_{L^{p}(Q_k)}^2\Big)^{1/2},$$
where $\R^{N}$ is covered by cubes $Q_k$, centered at $k\in\Z^{N}$ and having unitary edges parallel to the axes so that $\R^{N}=\bigcup_{k\in\Z^{N}}\overline{Q_k}$ and $Q_k\cap Q_l=\emptyset$ for $k\neq l$. 
Using the embedding of $H^{2\sigma}(Q_k)$ into $L^{p}(Q_k)$, we obtain by \cite[Lemma 2.4]{ACDRB04}
$$\abs{\int_{\R^{N}}m\phi\psi dx}\leq c\|m\|_{L^r_U(\R^{N})}\|\psi\|\Big(\sum_{k\in\Z^{N}}\|\phi\|_{H^{2\sigma}(Q_k)}^2\Big)^{1/2}\leq c\|m\|_{L^{r}_{U}(\R^{N})}\|\phi\|_{H^{2\sigma}(\R^{N})}\|\psi\|.$$
By the Riesz theorem there exists a unique $\phi^{*}\in L^2(\R^{N})$ such that $m\phi=\phi^*$ for $\phi\in H^{2\sigma}(\R^{N})$ and 
$$\|m\phi\|=\|\phi^{*}\|=\sup_{\|\psi\|=1}\abs{\int_{\R^{N}}\phi^{*}\psi dx}=\sup_{\|\psi\|=1}\abs{\int_{\R^{N}}m\phi\psi dx}\leq c\|m\|_{L^{r}_{U}(\R^{N})}\|\phi\|_{H^{2\sigma}(\R^{N})},$$
which proves \eqref{e:A0}.
\end{proof} 

Therefore, taking $\beta_0=0$ and $\alpha_0\in[\max\{\frac{N}{2r},\frac{1}{2}\},1)$ such that $\alpha_0>\frac{N}{2r}$ if $r=2$, by Lemma~\ref{lem:PROPERTY} 
and the fact that $\calf_{02}$ is a bounded linear operator from $L^2(\R^{N})$ into itself, we get
\begin{equation*}
\|\calf_{0}(\phi)-\calf_{0}(\psi)\|_{E^{\beta_0}}\leq c(\|m\|_{L^r_U(\R^N)}+1)\|\phi-\psi\|_{E^{\alpha_0}},\ \phi,\psi\in E^{\alpha_0}=H^{2\alpha_0}(\R^{N}).
\end{equation*}
Hence \eqref{e:EPSREG} holds with $\rho_0=1$ and in \eqref{e:WAR1EQUIV} we have $\omega_0=1-\alpha_0>0$.
 
Following \cite[Lemma B.1]{CHRB2012c}, \cite[Lemma 3.1]{CHRB2012a} 
by \eqref{e:CONDF0}, \eqref{e:CONDLIPF0}
there exists a decomposition 
\begin{equation}\label{e:DECOMPOSITION}
f_i(x,s)=f_{i1}(x,s)+f_{i2}(x,s),\ x\in \R^N,\ s\in\R,
\end{equation} 
such that $f_{i1}(x,0)=f_{i2}(x,0)=0$,
$f_{i1}\colon\R^N\times\R\to\R$ is globally Lipschitz w.r.t. the second variable and for some $c>0$
\begin{equation}\label{e:GROWTHF02}
|f_{i2}(x,s_1)-f_{i2}(x,s_2)|\leq c|s_1-s_2|(|s_1|^{\rho_i-1}+|s_2|^{\rho_i-1}),\ x\in\R^{N},\ s_1, s_2\in\R.
\end{equation}

Since our goal is to apply Theorem~\ref{thm:EXTRAPOLATED} and find $\gamma$-solutions of 
\eqref{e:ACPVCH} with $u_0\in H^1(\R^{N})$, in the lemma below we examine the admissible growth rate $\rho\geq 1$ such that
\begin{equation}\label{e:GROWTHH}
|h(x,s_1)-h(x,s_2)|\leq c|s_1-s_2|(|s_1|^{\rho-1}+|s_2|^{\rho-1}),\ x\in\R^{N},\ s_1, s_2\in\R,
\end{equation}
gives rise to a Nemytskii operator between suitable spaces $H^{2\alpha}(\R^{N})$ and $H^{2\beta}(\R^{N})$.

\begin{lem}\label{lem:SUITABLERHO}
Assume that $h\colon\R^{N}\times\R\to\R$ satisfies \eqref{e:GROWTHH} with $\rho\geq 1$ and let $\frac{1}{2}\leq \alpha<1$ and $-\frac{N}{4}<\beta\leq 0$, $\beta>\alpha-1$ satisfy
\begin{equation}\label{e:RESTRICTIONS}
\rho\geq 1-\tfrac{4\beta}{N},\quad\rho<\tfrac{\beta+\frac{1}{2}}{\alpha-\frac{1}{2}}\ \text{ if }\ \tfrac{1}{2}<\alpha<1\ \text{ and }\ \rho\leq\tfrac{\frac{N}{4}-\beta}{\frac{N}{4}-\alpha}\ \text{ if }\ \alpha<\tfrac{N}{4}.
\end{equation}
Then $h$ gives rise to the Nemytskii operator $\calh\colon H^{2\alpha}(\R^{N})\to H^{2\beta}(\R^{N})$ such that 
\begin{equation}\label{e:EPSREGH}
\|\calh(\phi)-\calh(\psi)\|_{H^{2\beta}(\R^{N})}\leq c\|\phi-\psi\|_{H^{2\alpha}(\R^{N})}(\|\phi\|_{H^{2\alpha}(\R^{N})}^{\rho-1}+\|\psi\|_{H^{2\alpha}(\R^{N})}^{\rho-1})
\end{equation}
holds for $\phi,\psi\in H^{2\alpha}(\R^{N})$, $0\leq\alpha-\beta<1$ and for $\gamma=\frac{1}{2}$ we have $\beta\leq\gamma\leq\alpha$ and
\begin{equation}\label{e:W1}
1+\beta-\alpha\rho+(\rho-1)\gamma>0.
\end{equation}
\end{lem}

\begin{proof}
If $-\frac{N}{4}<\beta\leq 0$ then by the Sobolev type embedding (see \cite[2.8.1/15]{triebel})
\begin{equation*}
L^q(\R^{N})\emb H^{2\beta}(\R^{N})\ \text{ for }q=\tfrac{2N}{N-4\beta}\in(1,2]
\end{equation*}
we infer from \eqref{e:GROWTHH} and the H\"older inequality that 
\begin{equation*}
\|\calh(\phi)-\calh(\psi)\|_{H^{2\beta}(\R^{N})}\leq 
c\|\phi-\psi\|_{L^{q\rho}(\R^{N})}(\|\phi\|^{\rho-1}_{L^{q\rho}(\R^N)}+\|\psi\|_{L^{q\rho}(\R^{N})}^{\rho-1}).
\end{equation*}
We consider $\gamma=\frac{1}{2}\leq\alpha<1$ such that
$H^{2\alpha}(\R^{N})$ embeds into $L^{q\rho}(\R^{N})$,
which holds if $q\rho\geq 2$ and $2\alpha-\frac{N}{2}\geq-\frac{N}{q\rho}$. This yields \eqref{e:EPSREGH}
for $\frac{1}{2}\leq \alpha<1$ and $-\frac{N}{4}<\beta\leq 0$ such that 
$\beta\geq\tfrac{N}{4}-\tfrac{N\rho}{4}$ and $\alpha\geq\tfrac{N}{4}-\tfrac{N}{4\rho}+\tfrac{\beta}{\rho}$.
To satisfy \eqref{e:W1} we also need
$\rho(\alpha-\frac{1}{2})<\beta+\frac{1}{2}$,
hence we get the restrictions in \eqref{e:RESTRICTIONS}.
\end{proof}

Taking \eqref{e:DECOMPOSITION} and \eqref{e:GROWTHF02} into account, we choose
$\beta_i=0$ for $i=1,\ldots,n$, and consider Nemytskii operators $\calf_i$ corresponding to $f_{i}$ acting from $E^{\alpha_i}$ to $E^0$ for some $\frac{1}{2}\leq\alpha_i<1$.
In Lemma~\ref{lem:SUITABLERHO} we derived conditions under which $\calf_i$, $i=1,\ldots,n$, satisfy \eqref{e:EPSREG} and \eqref{e:WAR1EQUIV} with $\rho_i\geq 1$ and $\gamma=\frac{1}{2}$, that is,
$$1\leq\rho_i<\tfrac{1}{2\alpha_i-1}\text{ if }\alpha_i>\tfrac{1}{2}\ \text{ and }\ 1\leq\rho_i\leq\tfrac{N}{N-4\alpha_i}\text{ if }\alpha_i<\tfrac{N}{4}.$$
Thus, if $N=1,2$ then we can have $\rho_i\geq 1$ arbitrary and choose $\alpha_i\in\bigl[\frac{1}{2},\frac{1}{2}(1+\frac{1}{\rho_i})\bigr)$, whereas if $N\geq 3$ then we can have $1\leq\rho_i<\frac{N+2}{N-2}$ and choose $\alpha_i\in\bigl[\max\{\frac{N}{4}(1-\frac{1}{\rho_i}),\frac{1}{2}\},\frac{1}{2}(1+\frac{1}{\rho_i})\bigr)$ for $i=1,\ldots,n$.

Since the operator $-\kappa(\delta+\kappa-\Delta)^{-1}$ is bounded linear in $E^0=L^2(\R^{N})$, \eqref{e:EPSREG} for $\calf_{n+1+i}$ follows from the corresponding property for $\calf_i$ with  $\beta_{n+1+i}=\beta_{i}=0$, $\alpha_{n+1+i}=\alpha_{i}$ and $\rho_{n+1+i}=\rho_{i}$ for $i=0,\ldots,n$. This shows that 
$\beta_i=0$, $\alpha_i\in[\frac{1}{2},1)$ for $i\in\cali$ and all assumptions of Theorem~\ref{thm:EXTRAPOLATED} are satisfied.

\begin{thm}
Under conditions \eqref{e:F}--\eqref{e:CONDRHO},
the abstract Cauchy problem \eqref{e:ACPVCH} corresponding to the modified viscous Cahn-Hilliard problem  \eqref{e:mvCH}, \eqref{e:INIT} 
possesses for any $u_0\in H^1(\R^{N})$  a unique $\gamma$-solution $u$, with $\gamma=\frac{1}{2}$, defined on the maximal interval of existence $[0,\tau_{u_0})$, that is,
$$u\in C([0,\tau_{u_0});H^1(\R^{N}))\cap C((0,\tau_{u_0});H^{2}(\R^{N}))\cap\bigcap_{T\in(0,\tau_{u_0})}\bigcap_{i=0}^{n} \mathcal{L}^\infty_{\alpha_i-\frac{1}{2}}((0,T];H^{2\alpha_i}(\R^{N})),$$
$$\dot{u}\in C((0,\tau_{u_0});H^{2_{-}}(\R^{N})),$$
which satisfies Duhamel's formula
\begin{equation*}
u(t)=S(t)u_0+\sum_{i=0}^{2n+1}\int_{0}^{t}S(t-s)\calf_i(u(s))ds,\ t\in[0,\tau_{u_0}),
\end{equation*}
and in $L^2(\R^{N})$ the differential equation
$$\dot{u}(t)+A u(t)=\sum_{i=0}^{2n+1}\calf_i(u(t)),\ t\in(0,\tau_{u_0}).$$ 
Moreover, if $\tau_{u_0}<\infty$ then
\begin{equation*}
\lim_{t\to\tau_{u_0}^{-}}\norm{u(t)}_{H^1(\R^{N})}=\infty.
\end{equation*}	
Furthermore, for any $0<\theta<\frac{1}{2}$ we have
\begin{equation*}
\lim_{t\to0^{+}}t^{\theta}\norm{u(t)}_{H^{1+2\theta}(\R^{N})}=0.
\end{equation*}
\end{thm}
 
\subsection{$2m$-th order equation in cell population dynamics}

In recent papers \cite{EV24,EV25} Efendiev and Vougalter considered a one-dimensional Cauchy problem for a $2m$-th order equation arising from mathematical biology and concerning the dynamics of cell density $u$ as a function of the cell genotype and time. The model featured a drift term $bu_x$, a~growth term $au$ and a nonlocal term $\int_\R G(x-y) f_1(y,u(y,t))dy$ with globally Lipschitz function $f_1$ and integrable function $G$. We generalize the problem to dimension $N$, allow the function $f_1$ to grow with rate $\rho_1\geq 1$ and model the growth of the cell population by a function $f_2$ with growth exponent $\rho_2\geq 1$. In order to simplify the argument, we do not consider the drift term here.       

More precisely, given $m\in\N$, we consider a Cauchy problem for
\begin{equation}\label{e:CELL}
\dot{u}+(-\Delta)^m u=(G\star f_1(\cdot,u))(x)+f_2(x,u),\ t>0,\ x\in\R^{N},
\end{equation}
where the given kernel $G\in L^{1}(\R^{N})$ does not vanish identically and the reaction terms $f_i\colon\R^N\times \R\to \R$, $i\in\cali=\{1,2\}$, are such that
\begin{equation}\label{e:CONDF2}
f_i(x,0)=0,\ x\in \R^N,
\end{equation}
\begin{equation}\label{e:CONDLIPF2}
|f_{i}(x,s_1)-f_{i}(x,s_2)|\leq c_0 |s_1-s_2|(1+|s_1|^{\rho_i-1} +|s_2|^{\rho_i-1}), \ x\in\R^{N},\ s_1, s_2\in \R,
\end{equation}
with a certain positive constant $c_0$ and $\rho_i\geq 1$, $i\in\cali$. 
We are going to treat $G\star f_1(\cdot,u)$ and $f_2(\cdot,u)$
as Nemytskii operators in the scale of Lebesgue spaces and compute the constraints for $\rho_i$'s so that for any $u_0\in L^\gamma(\R^{N})$, with given $\gamma\in[1,\infty]$, there exists a unique $\gamma$-solution to Duhamel's formula corresponding to \eqref{e:CELL}. 

For this purpose, we consider the $m$-harmonic heat semigroup 
$$S_m(t)u_0=K^m_t\star u_0,\ u_0\in \cals(\R^{N}),\ t>0,$$
where $K^m_t$ denotes the $m$-harmonic heat kernel (see \cite[Example VI.1]{EZ})
$$K^m_t(x)=(2\pi)^{-N}\int_{\R^{N}}e^{\mathbbm{i}z\cdot x-t|z|^{2m}}dz,\ x\in\R^{N},\ t>0.$$
The semigroup extends uniquely to $L^2(\R^{N})$ and $\{S_m(t)\}_{t>0}$ in $L^2(\R^{N})$ is generated by the polyharmonic operator $A=-(-\Delta)^m$ with $\dom(A)=H^{2m}(\R^{N})$. The semigroup also extends uniquely to any $L^p(\R^{N})$, $p\in[1,\infty]$, and smooths to any $L^q(\R^{N})$ with $q\geq p$. Indeed, if we consider $L^p(\R^{N})$ spaces with regularity index $\reg(p)=-\frac{N}{2mp}$ for $p\in\J=[1,\infty]$, then for $p\in\J_0=[1,\infty]$ and $p\leq q$ we have $p\leadsto q$ (see \cite[(4.4)]{CHQRB17}) and 
$$t^{\dif(p,q)}\norm{S_m(t)}_{\mathcal{L}(L^p(\R^{N});L^q(\R^{N}))}\leq M_m(p,q,\tau)\text{ for }0<t\leq\tau,$$
with $\dif(p,q)=\reg(q)-\reg(p)\geq 0$.

If $\calf_1$ is the Nemytskii operator corresponding to $G\star f_1(\cdot,u)$ on $L^{p_1}(\R^{N})$ with $p_1\in[\rho_1,\infty]$
and $\calf_2$ is the Nemytskii operator corresponding to $f_2$ on $L^{p_2}(\R^{N})$ with $p_2\in[\rho_2,\infty]$,
then we have $\calf_i\colon L^{p_i}(\R^{N})\to L^{q_i}(\R^{N})$ with $q_i=\frac{p_i}{\rho_i}$ for $i\in\cali$ and by the Young and H\"older inequality we get for $i\in\cali$
$$\|\calf_i(\phi)-\calf_i(\psi)\|_{L^{q_i}(\R^{N})}\leq c\|\phi-\psi\|_{L^{p_i}(\R^{N})}(1+\|\phi\|_{L^{p_i}(\R^{N})}^{\rho_i-1}+\|\psi\|_{L^{p_i}(\R^{N})}^{\rho_i-1}),\ \phi,\psi\in L^{p_i}(\R^{N}).$$
Our aim is to apply Corollary~\ref{cor:EXIST} to find a~$\gamma$-solution of Duhamel's formula with $\gamma\in[1,\infty]$
defined up to the maximal time of existence, 
\begin{equation}\label{e:DUHCELL}
u(t)=S_m(t)u_0+\sum_{i\in\cali}\int_{0}^{t}S_m(t-s)\calf_i(u(s))ds\ \text{ for }t\in (0,\tau_{u_0}),
\end{equation}
and later refine its regularity via Lemma~\ref{lem:BETTERREGULARITY}.

Note that $\{S_m(t)\}_{t>0}$ is a~semigroup on $L^{\gamma}(\R^{N})$ and on each $L^{q_i}(\R^{N})$ for $i\in\cali$. 
To accomplish our task, we need to know that \eqref{e:SETUP1}, \eqref{e:SETUP2}, \eqref{e:CONDALPHAJ}, \eqref{e:TOGAMMA} and \eqref{e:POSITIVEOMEGA} are satisfied. Note that \eqref{e:POSITIVEOMEGA} means that 
\begin{equation*}
\tfrac{N}{2m}(\rho_i-1)<\gamma\ \text{ for }i\in\cali,
\end{equation*} 
whereas \eqref{e:SETUP2} and \eqref{e:TOGAMMA} hold if
\begin{equation}\label{e:10NEW}
\max\{\gamma,\rho_i\}\leq p_i\leq\gamma\rho_i\ \text{ for }i\in\cali.
\end{equation}
The choice of $p_1$ and $p_2$ is further limited by \eqref{e:SETUP1} and \eqref{e:CONDALPHAJ}, that is,
\begin{equation}\label{e:20}
\tfrac{\rho_1}{p_1}<\tfrac{2m}{N}+\tfrac{1}{p_2}\ \text{ and }\ \tfrac{\rho_2}{p_2}<\tfrac{2m}{N}+\tfrac{1}{p_1}.
\end{equation} 

\begin{thm}\label{thm:CELL}
Let $m\in\N$, $G\in L^1(\R^{N})$, $\gamma\in[1,\infty]$ and $f_i\colon\R^{N}\times\R\to\R$, $i\in\cali=\{1,2\}$, satisfy \eqref{e:CONDF2}, \eqref{e:CONDLIPF2} with
\begin{equation}\label{e:RHOBOUNDS}
1\leq\rho_i<\tfrac{2m}{N}\gamma+1\ \text{ for }i\in\cali.
\end{equation}
Then for any $u_0\in L^\gamma(\R^{N})$ there exists a~unique $\gamma$-solution $u(\cdot,u_0)$ of \eqref{e:DUHCELL} defined on $(0,\tau_{u_0})$ such that for $T\in(0,\tau_{u_0})$
\begin{equation}\label{e:HOWREGULARCELL}
u(\cdot,u_0)\in C((0,\tau_{u_0});L^\gamma(\R^{N}))\cap\bigcap_{i\in\cali}C((0,\tau_{u_0});L^{\gamma\rho_i}(\R^{N}))\cap\mathcal{L}^\infty_{\frac{N}{2m\gamma}(1-\frac{1}{\rho_i})}((0,T];L^{\gamma\rho_i}(\R^{N}))
\end{equation}
and
\begin{equation}\label{e:LIMSUPCELL}
\lim_{t\to\tau_{u_0}^{-}}\norm{u(t,u_0)}_{L^{\gamma}(\R^{N})}=\infty\ \text{ provided that }\tau_{u_0}<\infty.
\end{equation}
Moreover, for $\gamma\in[1,\infty)$ the solution $u(\cdot,u_0)$ extends continuously to a function in $C([0,\tau_{u_0});L^\gamma(\R^{N}))$ with $u(0)=u_0$.

\noindent 
Furthermore, 
if $\rho=\max\{\rho_1,\rho_2\}$ then for any $q\in[\gamma\rho,\infty]$ we have
\begin{equation}\label{e:HOWREGULARCELL2}
\begin{split}
&u(\cdot,u_0)\in C((0,\tau_{u_0});L^q(\R^{N}))\cap\bigcap_{T\in(0,\tau_{u_0})}\mathcal{L}^\infty_{\frac{N}{2m}(\frac{1}{\gamma}-\frac{1}{q})}((0,T];L^q(\R^{N})).
\end{split}
\end{equation}
\end{thm}

\begin{proof}
Taking $p_i=\gamma\rho_i$, $i\in\cali$, we see that \eqref{e:10NEW} and \eqref{e:20} are satisfied. Hence, by Corollary~\ref{cor:EXIST} we obtain the existence and uniqueness of $u$ satisfying \eqref{e:HOWREGULARCELL} and \eqref{e:LIMSUPCELL},
whereas the continuous extension at $t=0$ in $L^\gamma(\R^{N})$ for $\gamma\in[1,\infty)$ follows from the fact that the semigroup $\{S_m(t)\}_{t>0}$ is then strongly continuous.  

Note that if $\gamma=\infty$ then we already know that for $T\in(0,\tau_{u_0})$
$$u(\cdot,u_0)\in C((0,\tau_{u_0});L^\infty(\R^{N}))\cap L^\infty((0,T];L^\infty(\R^{N})).$$
Therefore, to prove \eqref{e:HOWREGULARCELL2}, let $\rho=\max\{\rho_1,\rho_2\}$ and consider $\gamma\in[1,\infty)$. We will show that $u\in C((0,\tau_{u_0});L^q(\R^{N}))\cap \mathcal{L}_{\dif(\gamma,q)}^{\infty}((0,T];L^q(\R^{N}))$ for any $q\in(\gamma\rho,\infty]$ and $T\in(0,\tau_{u_0})$. 

First let $q\in(\gamma\rho,\infty)$ and divide $[\gamma\rho,q]$ into $k$ subintervals by points $\delta_n=\gamma\rho+\frac{n}{k}(q-\gamma\rho)$, $n=0,\ldots,k$, where $k\in\N$ is so large that
\begin{equation}\label{e:LARGEK}
\tfrac{q-\gamma\rho}{(k\gamma\rho+q-\gamma\rho)\gamma\rho}<\tfrac{2m}{N}-\tfrac{\rho-1}{\gamma\rho}.
\end{equation}
We apply Lemma~\ref{lem:BETTERREGULARITY} and Remark~\ref{rem:ITERATION} (ii) successively with each $\delta_n$. We already know that $u\in C((0,\tau_{u_0});L^{\delta_0}(\R^{N}))\cap\mathcal{L}_{\dif(\gamma,\delta_0)}^\infty((0,T];L^{\delta_0}(\R^{N}))$ for $T\in(0,\tau_{u_0})$. Suppose that 
$$u\in C((0,\tau_{u_0});L^{\delta_n}(\R^{N}))\cap\mathcal{L}_{\dif(\gamma,\delta_n)}^\infty((0,T];L^{\delta_n}(\R^{N})),\ T\in(0,\tau_{u_0}),$$ 
for some $n\in\{0,\ldots,k-1\}$. Note that we have a counterpart of \eqref{e:EPSREG2} with $\alpha_i=\delta_n$ and $\beta_i=\frac{\delta_n}{\rho_i}$ for $i\in\cali$. Then the counterpart of \eqref{e:WEAKERSETUP2} holds trivially, whereas
the counterpart of \eqref{e:NONNEGATIVEOMEGA} means $\frac{N}{2m}(\rho_i-1)\leq\gamma$  for all $i\in\cali$
and is also trivially guaranteed by \eqref{e:RHOBOUNDS}. Moreover, \eqref{e:ALPHAJDELTA} holds with $\delta=\delta_{n+1}$
and then \eqref{e:THETA} is satisfied, since by \eqref{e:LARGEK} we have
$$\tfrac{\rho}{\delta_n}-\tfrac{1}{\delta_{n+1}}=\tfrac{\rho-1}{\delta_n}+\tfrac{q-\gamma\rho}{k\delta_{n+1}\delta_{n}}\leq\tfrac{\rho-1}{\gamma\rho}+\tfrac{q-\gamma\rho}{(k\gamma\rho+q-\gamma\rho)\gamma\rho}<\tfrac{2m}{N}.$$
This shows that after $k$ steps $u\in C((0,\tau_{u_0});L^q(\R^{N}))\cap\mathcal{L}_{\dif(\gamma,q)}^\infty((0,T];L^{q}(\R^{N}))$ for $T\in(0,\tau_{u_0})$ with an arbitrary $q\in(\gamma\rho,\infty)$.
Choosing initially $q>\max\{\frac{N}{2m},\gamma\}\rho$, in the subsequent application of Lemma~\ref{lem:BETTERREGULARITY} with $\delta=\infty$ we conclude that 
$$u\in C((0,\tau_{u_0});L^\infty(\R^{N}))\cap\mathcal{L}_{\frac{N}{2m\gamma}}^\infty((0,T];L^{\infty}(\R^{N}))\ \text{ for }T\in(0,\tau_{u_0}),$$
which ends the proof.
\end{proof}

We now consider another $2m$-th order equation with $G\star f_1(\cdot,u)$ in \eqref{e:CELL} replaced by the Hartree type nonlinearity
$(G\star|u|^2)u$, that is,
\begin{equation}\label{e:CELL2}
\dot{u}+(-\Delta)^m u=(G\star|u|^2)u+f(x,u),\ t>0,\ x\in\R^{N},
\end{equation}
with a given potential $G\in L^{1}(\R^{N})$ and $f\colon\R^N\times \R\to \R$ which satisfies
\begin{equation}\label{e:CONDF3}
f(x,0)=0,\ x\in \R^N,
\end{equation}
\begin{equation}\label{e:CONDLIPF3}
|f(x,s_1)-f(x,s_2)|\leq c_0 |s_1-s_2|(1+|s_1|^{\rho-1} +|s_2|^{\rho-1}), \ x\in\R^{N},\ s_1, s_2\in \R,
\end{equation}
with certain constants $c_0>0$ and $\rho\geq 1$.

In \cite[Lemma 6.1]{CHQRB17}  it was shown that the Nemytskii operator $\calf_1\colon L^{4}(\R^{N})\to L^{\frac{4}{3}}(\R^{N})$  corresponding to $(G\star|u|^2)u$ satisfies for $\phi,\psi\in L^4(\R^N)$
\begin{equation*}
\|\calf_1(\phi)-\calf_1(\psi)\|_{L^\frac{4}{3}(\R^N)}\leq \tfrac{3}{2}\|G\|_{L^1(\R^N)}\|\phi-\psi\|_{L^4(\R^N)}(\|\phi\|^2_{L^4(\R^N)}
+\|\psi\|^2_{L^4(\R^N)}).
\end{equation*}
Recall also that for the Nemytskii operator $\calf_2\colon L^p(\R^N)\to L^{\frac{p}{\rho}}(\R^{N})$ corresponding to $f$ with $p\in[\rho,\infty]$ we have
$$\|\calf_2(\phi)-\calf_2(\psi)\|_{L^{\frac{p}{\rho}}(\R^{N})}\leq c\|\phi-\psi\|_{L^{p}(\R^{N})}(1+\|\phi\|_{L^{p}(\R^{N})}^{\rho-1}+\|\psi\|_{L^{p}(\R^{N})}^{\rho-1}),\ \phi,\psi\in L^{p}(\R^{N}).$$
Verifying the assumptions of Corollary~\ref{cor:EXIST} again, we see that \eqref{e:POSITIVEOMEGA} means that
\begin{equation}\label{e:WARG1}
\gamma>\tfrac{N}{m}\ \text{ and }\ \gamma>\tfrac{N}{2m}(\rho-1),
\end{equation}
\eqref{e:SETUP2} and \eqref{e:TOGAMMA} hold if
\begin{equation}\label{e:WARG2}
\gamma\in[\tfrac{4}{3},4]\ \text{ and }\ \max\{\gamma,\rho\}\leq p\leq\gamma\rho,
\end{equation}
whereas \eqref{e:SETUP1} and \eqref{e:CONDALPHAJ} further require that
\begin{equation}\label{e:WARP}
\tfrac{3}{4}<\tfrac{2m}{N}+\tfrac{1}{p}\ \text{ and }\ \tfrac{\rho}{p}<\tfrac{2m}{N}+\tfrac{1}{4}.
\end{equation}
Note that, having \eqref{e:WARG1} and \eqref{e:WARG2},  if $p\leq 4$ then the left inequality in \eqref{e:WARP} holds, while if $p>4$ then the right one holds.

Combining the above restrictions, we get the following result.

\begin{thm}\label{thm:CELL2}
Let $m\in\N$, $m>\frac{N}{4}$ and $\gamma\in[\frac{4}{3},4]$ be such that $\gamma>\frac{N}{m}$. 
Assume that $G\in L^1(\R^{N})$ and $f\colon\R^{N}\times\R\to\R$  satisfies \eqref{e:CONDF3} and \eqref{e:CONDLIPF3} with 
\begin{equation*}
1\leq\rho<\tfrac{2m}{N}\gamma+1.
\end{equation*}
If $p\in[\max\{\gamma,\rho\},\gamma\rho]$ is such that
either
$$p\leq 4\ \text{ and }\ \tfrac{1}{p}\in[\tfrac{1}{4},\tfrac{1}{\rho}(\tfrac{2m}{N}+\tfrac{1}{4}))$$
or 
$$p>4\ \text{ and }\ \tfrac{1}{p}\in(\tfrac{3}{4}-\tfrac{2m}{N},\tfrac{1}{4}),$$
then for any $u_0\in L^\gamma(\R^{N})$ there exists a~unique $\gamma$-solution $u(\cdot,u_0)$ of \eqref{e:DUHCELL}, corresponding to \eqref{e:CELL2}, defined on $[0,\tau_{u_0})$ such that for any $T\in(0,\tau_{u_0})$
\begin{equation*}
\begin{split}
u(\cdot,u_0)\in C([0,\tau_{u_0});L^\gamma(\R^{N}))&\cap C((0,\tau_{u_0});L^{p}(\R^{N}))\\
&\cap \mathcal{L}^\infty_{\frac{N}{2m}(\frac{1}{\gamma}-\frac{1}{p})}((0,T];L^{p}(\R^{N}))\cap \mathcal{L}^\infty_{\frac{N}{2m}(\frac{1}{\gamma}-\frac{1}{4})}((0,T];L^{4}(\R^{N}))
\end{split}
\end{equation*}
and \eqref{e:LIMSUPCELL} holds.
\end{thm}

In Theorem~\ref{thm:CELL2}, for $m=1$, $N=2$ and $\gamma\in(2,4]$ we have $1\leq\rho<\gamma+1$ and we can choose any $p\in[\max\{\gamma,\rho\},\gamma\rho]$.
In particular, $\calf_1$ and $\calf_2$ can have different domains and growth rates.

\subsection{Semilinear Schr\"odinger equation}

We first consider the linear Schr\"odinger semigroup 
$$S(t)u_0=(4\pi\mathbbm{i} t)^{-\frac{N}{2}}e^{\frac{\mathbbm{i}|\cdot|^2}{4t}}\star u_0,\ u_0\in \cals(\R^{N}),\ t>0,$$
in the scale of Lebesgue spaces $E^p=L^p(\R^{N})$ with regularity index $\reg(p)=-\frac{N}{2p}$ for $p\in[1,\infty]$. 
It is known (see \cite[Proposition 2.2.3]{CAZ}) that if $p\in[1,2]$, then $p\leadsto p'=\frac{p}{p-1}\in[2,\infty]$ and with $\dif(p,p')=\reg(p')-\reg(p)\geq 0$ we have
\begin{equation}\label{e:SMOOTHSCHR}
t^{\dif(p,p')}\norm{S(t)}_{\mathcal{L}(L^{p}(\R^{N});L^{p'}(\R^{N}))}\leq (4\pi)^{-\dif(p,p')},\ t>0.
\end{equation}
We can still apply the theory of $\gamma$-solutions from Theorem~\ref{thm:MAIN1} to the abstract semilinear Schr\"odinger equation 
\begin{equation}\label{e:SEMSCHR}
\dot{u}-\mathbbm{i}\Delta u=\sum_{i\in\cali}\calf_i(u),\ t>0,\ x\in\R^{N},
\end{equation}
with $\calf_i$ satisfying \eqref{e:EPSREG}, \eqref{e:EPSREG2} with $\rho_i\geq 1$ on the scale of Lebesgue spaces by looking for $\gamma$-solutions of Duhamel's formula
\begin{equation}\label{e:DUHSCHR}
u(t)=S(t)u_0+\sum_{i\in\cali}\int_{0}^{t}S(t-s)\calf_i(u(s))ds\ \text{ for }\ t\in (0,\tau],
\end{equation}
where $u_0\in L^\gamma(\R^{N})$ for some $\gamma\geq 1$.
However, due to assumptions \eqref{e:SETUP1}, \eqref{e:SETUP2} and scarce available smoothing properties \eqref{e:SMOOTHSCHR} of the family $\{S(t)\}_{t>0}$, we need to have  $\alpha_i=\gamma'$, $\beta_i=\gamma$ for $i\in\cali$ and $\gamma\in\J_0=[1,2]$. This, in turn, implies by the Cauchy inequality  that 
the right-hand side of \eqref{e:SEMSCHR}
can be treated as a single nonlinearity $\calf\colon L^{\gamma'}(\R^N)\to L^{\gamma}(\R^{N})$ such that for $\phi,\psi\in L^{\gamma'}(\R^{N})$
\begin{equation}\label{e:EPSREGSCHR}
\norm{\calf_i(\phi)-\calf_i(\psi)}_{L^{\gamma}(\R^{N})}\leq c_0\norm{\phi-\psi}_{L^{\gamma'}(\R^{N})}(1+\norm{\phi}_{L^{\gamma'}(\R^{N})}^{\rho-1}+\norm{\psi}_{L^{\gamma'}(\R^{N})}^{\rho-1})
\end{equation}
holds with $c_0>0$ and $\rho=\max\limits_{i\in\cali}\rho_i\geq 1$. Thus, applying Theorem~\ref{thm:MAIN1}, we need to verify
\eqref{e:CONDALPHAJ} and \eqref{e:POSITIVEMUIORAZOMEGAINTRO}, which both hold provided that 
\begin{equation}\label{e:CONDFORRHO}
N\bigl(\tfrac{1}{\gamma}-\tfrac{1}{2}\bigr)=\reg(\gamma')-\reg(\gamma)<\tfrac{1}{\rho}. 
\end{equation}

\begin{thm}
Let  $\gamma\in[1,2]$ and $\calf_i\colon L^{\gamma'}(\R^{N})\to L^{\gamma}(\R^{N})$, $i\in\cali$, satisfy \eqref{e:EPSREGSCHR} with $\rho\geq 1$ such that \eqref{e:CONDFORRHO} holds.
Then for any $u_0\in L^\gamma(\R^{N})$ there exists a~locally unique $\gamma$-solution $u(\cdot,u_0)$ of \eqref{e:DUHSCHR} defined on $(0,\tau]$ such that
\begin{equation*}
u(\cdot,u_0)\in\mathcal{L}^\infty_{N(\frac{1}{\gamma}-\frac{1}{2})}((0,\tau];L^{\gamma'}(\R^{N})).
\end{equation*}
If $\gamma=2$ and $\rho\geq 1$ is arbitrary, then $u(\cdot,u_0)\in C([0,\tau_{u_0});L^2(\R^{N}))$
with the maximal time of existence $\tau_{u_0}$ and
\begin{equation*}
\lim_{t\to\tau_{u_0}^{-}}\norm{u(t,u_0)}_{L^2(\R^{N})}=\infty\text{ provided that }\tau_{u_0}<\infty.
\end{equation*}	
\end{thm} 

Recall from \cite[Lemma 6.1]{CHQRB17} that among admissible nonlinearities one finds the Hartree function with potential $G\in L^{p_0}(\R^{N})$ for $p_0\geq 1$, 
$$\calf(u)=(G\star|u|^2)u,\ u\in L^{\gamma'}(\R^{N}),$$
where $\gamma=\frac{4p_0}{2p_0+1}$ and $\frac{3N}{4}<p_0\leq\infty$, since then \eqref{e:EPSREGSCHR} holds with $\rho=3$ and the latter condition guarantees \eqref{e:CONDFORRHO}.

\section*{Statements and declarations}

\noindent 
\textbf{Author Contributions.} R. Czaja and M. Kania contributed equally to the manuscript.

\noindent 
\textbf{Funding.} No funding was received for conducting this study.

\noindent 
\textbf{Data Availability.} No datasets were generated or analysed during the current study.

\noindent 
\textbf{Conflicts of interest/Competing interests.} The authors have no relevant financial or non-financial interests to disclose.

\end{document}